\documentclass[11pt,a4paper]{article}

\usepackage[T1]{fontenc}
\usepackage[utf8]{inputenc}
\usepackage[english]{babel}
\usepackage{amsfonts, amssymb, amsmath, amsthm,mathbbol}
\usepackage[all]{xy}
\usepackage{enumerate}

\usepackage{graphicx}
\usepackage{psfrag}
\usepackage{geometry}
\usepackage[colorlinks=true,pdfstartview=FitV,linkcolor=blue,citecolor=green,urlcolor=black,filecolor=magenta]{hyperref}
\usepackage{verbatim}

\newtheorem{theorem}{Theorem}[section]
\newtheorem{definition}[theorem]{Definition}
\newtheorem{prop}[theorem]{Proposition}
\newtheorem{lemma}[theorem]{Lemma}
\newtheorem{cor}[theorem]{Corollary}
\newtheorem{rem}[theorem]{Remark}
\newtheorem{exam}[theorem]{Example}

\newcommand{\Aa}{{\mathcal A}}

\newcommand{\Gg}{{\mathcal G}}
\newcommand{\Hh}{{\mathcal H}}

\newcommand{\Kk}{{\mathcal K}}
\newcommand{\Ll}{{\mathcal L}}

\newcommand{\Oo}{{\mathcal O}}

\newcommand{\Rr}{{\mathcal R}}

\newcommand{\Tt}{{\mathcal T}}
\newcommand{\Uu}{{\mathcal U}}
\newcommand{\Vv}{{\mathcal V}}

\newcommand{\g}{{\mbox{\rm g}}}

\newcommand{\id}{{\mathbf 1}}

\newcommand{\CM}{{\mathbb C}}

\newcommand{\NM}{{\mathbb N}}

\newcommand{\PM}{{\mathbb P}}
\newcommand{\RM}{{\mathbb R}}

\newcommand{\TM}{{\mathbb T}}

\newcommand{\ZM}{{\mathbb Z}}

\newcommand{\HG}{{\mathfrak H}}

\newcommand{\diam}{\mbox{\rm diam}\,}      

\newcommand{\TR}{\text{\rm Tr}} 
\renewcommand{\phi}{\varphi}                  

\renewcommand{\tilde}{\widetilde}

\newcommand{\wee}{\, \tilde\wedge \,}       

\newcommand{\dtau}{d_\tau}        

\newcommand{\dsup}{d_{\text{\rm sup}}}   
\newcommand{\dinf}{d_{\text{\rm inf}}}   
\newcommand{\op}{\text{\rm op}}     

\newcommand{\prs}{p_{\text{\rm rs}}} 
\newcommand{\ppr}{p_{\text{\rm pr}}} 
\newcommand{\brs}{\beta_{\text{\rm rs}}} 
\newcommand{\bpr}{\beta_{\text{\rm pr}}} 

\newcommand{\pf}{\Lambda_\text{\rm {\tiny PF}}}	

\newcommand{\freq}{\textrm{freq}}		
\newcommand{\vol}{\textrm{vol}}		

\newcommand{\bup}{\overline{\beta}}
\newcommand{\blo}{\underline{\beta}}

\newcommand{\R}{\mathbb{R}}
\newcommand{\N}{\mathbb{N}}
\newcommand{\Z}{\mathbb{Z}}
\newcommand{\C}{\mathbb{C}}
\newcommand{\HS}{\mathfrak H}
\newcommand{\HE}{\mathcal H}
\newcommand{\Tv}{\Tt^{(0)}}
\newcommand{\Te}{\Tt^{(1)}}

\newcommand{\Gv}{\Gg^{(0)}}
\newcommand{\Ge}{\Gg^{(1)}}

\newcommand{\Rob}{\Rr}

\newcommand{\rtr}{\delta_{tr}}
\newcommand{\rlg}{\delta_{lg}}

\newcommand{\one}{1}
\newcommand{\OP}{\Omega_\Phi}

\DeclareMathOperator{\Ind}{ind}
\DeclareMathOperator{\coker}{coker}
\DeclareMathOperator{\Hom}{Hom}

\title{On the noncommutative geometry of tilings}
\author{Antoine Julien, Johannes Kellendonk, Jean Savinien}
\date{\today}

\begin{document}

\maketitle

\tableofcontents

\section{Introduction}\label{sec-intro}

Alain Connes' noncommutative geometry program is based on translating the ordinary notions of geometry into the framework of associative algebras in such a way that they make sense also
for noncommutative algebras. This is very natural from the point of view of quantum physics.
The basic objects are then no longer spaces but algebras.

Methods from noncommutative topology--a branch of noncommutative geometry--have been applied quite early to aperiodic tilings. Although noncommutative topology will be touched on only briefly in the last section of this article we will shortly describe the history of its applications to tilings of the Euclidean space. The start was made by Alain Connes himself and by Jean Bellissard.
In fact, one of the first examples of a ``noncommutative space'' which Connes discussed in his book \cite{Co-french} 
is a description of the most famous of all aperiodic tilings in the plane, the Penrose tilings, 
and Jean Bellissard proposed a $C^*$-algebraic approach to the description of aperiodic media \cite{Bel86}. 

The point of departure in a noncommutative theory is an associative algebra, often a $C^*$-algebra, which, in the context we consider here, should be somehow derived from a tiling or a point set. Connes' version of the $C^*$-algebra for the Penrose tilings is based on their substitution rule.  Robinson had already observed that the substitution rule can be used to describe the set of Penrose tilings modulo isometry as the set of $\{0,1\}$-sequences without consecutive $1$'s \cite{Grunbaum}, a description which Connes recognized as the set of paths on the Bratteli diagram which has inclusion graph $A_4$. Consequently, Connes assigned to the Penrose tiling the AF-algebra which is defined by the Bratteli diagram. 

Bellissard's $C^*$-algebra of an aperiodic medium is the crossed product algebra defined by a covariant family of Schr\"odinger operators associated with given local atomic potentials, a technique which had been devised for the study of disordered systems. Such families are representations of the crossed product algebra of a dynamical system whose underlying space is what Bellissard introduced as the {\em hull of the potential} and whose group action is given by translation of the potential in space (see \cite{BHZ} for a later review on the subject). 
The spatial arrangement of a medium may be described by a tiling or a point set and thus Bellissard's crossed product algebra may be considered as a version of a $C^*$-algebra for the tiling.

Later, Johannes Kellendonk provided a direct geometric construction of an algebra for a tiling, the so-called discrete tiling algebra. The algebra is the groupoid-$C^*$-algebra of an \'etale groupoid which arises when considering the translations which one can make in order to go from the center of one tile of the tiling to the center of another tile \cite{Kel95}. No Schr\"odinger operator is needed for this construction, neither a substitution rule. In its most rigorous form in can be derived from a purely combinatorial object, the inverse semigroup of the tiling, the topology being entirely determined by the order relation on the semigroup \cite{Kel97}.

Finally, 
Ian Putnam with his student Jared Anderson constructed an algebra for a tiling again in form of 
the crossed product algebra of a dynamical system. This dynamical system had been originally defined by Rudolph \cite{Rudolph}. Its underlying space is the
 {\em continuous hull of the tiling} and the algebra is referred to as the continuous tiling algebra.  
 The continuous hull  is closely related to the hull of the potential and often the same.
 
To summarise, there were several approaches to construct the basic object for the noncommutative geometry of tilings. The first was derived from a substitution, the second from a potential whose spatial repetition is governed by the tiling, the third and the fourth directly from the tiling. As it turned out, the second and the fourth algebra are essentially equal 
and can be seen as the stabilized version of the third algebra. From the point of view of noncommutative topology this means that the three latter approaches are equivalent. 
Connes' AF-algebra is closest to the third, the discrete tiling algebra, namely it is a (proper) subalgebra of the latter. The noncommutative topological invariants of the tiling algebra are richer than those of this AF-algebra. The AF-algebra should be regarded as describing the substitution symmetry of the tiling rather than the tiling itself.

Newer developments concern the noncommutative topology of tilings with infinite rotational symmetry, like the Pinwheel tilings \cite{Mou10}, of tilings in hyperbolic space \cite{OP11} and 
of combinatorial tilings \cite{Maria-phd,Maria1}. These developments will not be explained in this book.

All the above algebras are noncommutative $C^*$-algebras and come with interesting maximal commutative subalgebras. These are the algebra of continuous functions over the continuous hull, the algebra of continuous functions over an abstract transversal of that hull (often referred to as canonical transversal or discrete hull), and the commutative subalgebra of the AF-algebra which is the fixed point algebra of the latter under rotations and reflections. 

The next step in noncommutative topology is the investigation of the $K$-theory of the algebras.
This has been to a large extend discussed in the review article \cite{KP99}.
While it is quite simple to compute the $K$-groups for AF-algebras, it is less so for crossed products and the computation of the $K$-theory for the discrete or continuous tiling algebra (both yield the same answer) forms a substantial part of tiling theory. In fact, neglecting order the $K$-groups are isomorphic to cohomology groups (at least in low dimension with integral coefficients, and always with rational coefficients), and the latter are computable for certain classes of tilings, namely substitution tilings and almost canonical cut-and-project tilings. 
This is described in the chapter about cohomology of this book [\texttt{inside citation}]. 

Another part of noncommutative geometry is cyclic cohomology and the pairing between cyclic cohomology and $K$-theory. Here the older results cover only the pairing of the $K$-groups with very specific cyclic cocycles, like the trace, the noncommutative winding number and the noncommutative Chern character. We will have to say something more systematic about this pairing in the context of the algebra of functions over the discrete hull in the last section of this chapter.

We end this short survey on the noncommutative topology of tiling with the remark that both, noncommutative geometry
and the theory of tilings (aperiodic order) were strongly motivated and influenced by physics. The $K$-groups and their pairing with cyclic cocycles
have relevance to physics, namely to topological quantization, in particular in the gap-labelling \cite{BHZ, Bel86, BBG06},
the Integer Quantum Hall Effect \cite{BES-B94, KRS-B02} and other topological insulators, 
the pressure on the boundary \cite{Kel05} and the Levinson's theorem \cite{KR06, KR08, KR12, BS12}. These fascinating developments are also beyond the scope of this book. 
\bigskip

After this short description of the developments of the noncommutative topology of tilings we come to the topic of this article, namely the noncommutative {\em geometry} of tilings. 
Geometry is the investigation of a space through the measurement of length or distance between points. The fundamental object in geometry is thus a length or distance function. In noncommutative geometry it is hence a notion which allows one to construct the analog of such a function for (possibly noncommutative) algebras, the guiding principle being always the duality between $C^*$-algebras and topological spaces. Motivated by quantum physics Connes advocates that length should be a spectral quantity (we won't try to be precise on what that really means) and proposed the notion of a spectral triple as the fundamental object of noncommutative geometry. 

Leaving the precise definition for later (Definition~\ref{def-spectral-triple}) we can say that
a spectral triple is a representation of our favorite associative algebra on a Hilbert space together with a choice of a self-adjoint typically unbounded operator on that Hilbert space, the Dirac operator $D$, such that $[D,\cdot]$ can play the role of a derivative. 
The notion is modeled after Riemannian spin manifolds
and then cast into an axiomatic framework which we will partly recall further down. The subject of this chapter is to summarize what is known about spectral triples for tilings. This subject is wide open, in fact, no satisfying spectral triple has been constructed for the full tiling algebra (neither the discrete, nor the continuous one). What has been achieved so far are constructions for the maximal commutative subalgebras mentioned above. Though this seems disappointing at first sight, it has already let to new concepts for the study of tilings. 

What is a spectral triple good for? As we said already, it provides us with a noncommutative notion of distance, namely a pseudometric on the space of states of the algebra. In the case we consider here, a commutative tiling algebra, this yields in particular a distance function on the discrete or the continuous hull. This is an extra structure and we may ask: Does this distance function generate the topology? And if not, what does this say about the tiling?
A second ingredient a spectral triple defines is a meromorphic function, the so-called zeta function of the triple. The pole structure of the zeta-function defines various dimensions (dimension spectrum). In the commutative case these dimensions can be compared with Hausdorff or Minkowski dimensions which arise if the space is equipped with a metric.
A third object is a particular state on the algebra, the spectral state, which in the commutative case corresponds to a measure, the spectral measure. Hence one does not only have the elements of a differential calculus at hand--via the commutator with $D$--but also an integral. This allows one to define the fourth object, a quadratic form which may be interpreted as a Laplace operator.  
The fifth and final object we consider here is a $K$-homology class, an element of noncommutative topology which ought to be of fundamental importance for the underlying geometry.
Many more structures can be obtained from the spectral triple of an algebra, but we will not consider them here. Instead we refer to two rather comprehensive books on the subject \cite{Co94, G-BFV00}.

How does one construct spectral triples for the commutative algebra of a tiling? The basic idea, which goes back to Christiansen \& Ivan, is based on the pair spectral triple. A space\footnote{Here and in the following we mean by a spectral triple for a space a spectral triple for the (commutative) algebra of continuous functions over the space.}
 consisting of just two points allows for an essentially unique non-trivial spectral triple depending on one parameter only which can be interpreted as the distance between the two points. This is the so-called pair spectral triple. Now one approximates the space by a countable dense subset and declares in a hierarchical way which points are paired up to make a pair spectral triple. The actual spectral triple is then the direct sum of all pair spectral triples. There are of course a lot of choices to make along the way like, for instance, the parameter for the pair spectral triple. One of the surprising aspects of the construction is that it pays off to bundle up certain choices in a so-called choice function and to consider a whole family of spectral triples parametrised by the choice functions. Objects as the ones mentioned above and certain quantities derived from them can then be obtained either by taking extremal values or by averaging over the choice functions. 

In different disguises this approach has been considered for general compact metric spaces \cite{CI07,Palmer}, ultrametric Cantor sets \cite{PB09}, fractals \cite{GI03,GI05} and the spaces of tilings and subshifts. With the exception of the results on fractals we review these works in the unifying framework of approximating graphs with an emphasis on tiling and subshift spaces.
More precisely, we consider in Section~\ref{sec-concrete} the {\em spectral triple of a subshift} (Section~\ref{ssec-STsubshift}), the {\em ordinary transverse spectral triple of a tiling} (Section~\ref{ssec-STtrans}), the {\em transverse substitution spectral triple} of a substitution tiling (Section~\ref{ssec-trans}), the {\em longitudinal substitution spectral triple} of a prototile of a substitution tiling (Section~\ref{ssec-lgST}), and, combining the last to the {\em full substitution spectral triple} of a substitution tiling (Section~\ref{ssec-SThull}). The following sections are then devoted to the study of the above-mentioned objects which can be defined from the data of a spectral triple. 

Section~\ref{sec-zeta} is devoted to the zeta-function and the relation between its poles, various dimensions defined for metric spaces, and, most importantly, complexity exponents of subshifts and tilings.

In Section~\ref{sec-laplace} we review the results on Dirichlet forms and Laplacians for ultrametric Cantor sets, discrete hulls of tilings and the continuous hull of the substitution tiling. The first proposal of a Dirichlet form for the spectral triple of an ultrametric Cantor set has been made by Pearson and Bellissard \cite{PB09}. 
We discuss it in Section~\ref{ssec-PBLap}. It is in some sense not the canonical Dirichlet form, as
it involves an averaging over the choice functions and
integration is not defined w.r.t.\ the spectral state but rather with the ordinary operator trace.
It can be computed quite explicitly in the case of transversals of substitution tilings. The significance of this so-called Pearson--Bellissard Laplacian has still to be understood. Given that the abstract transversal of a tiling is a Cantor set which can be provided with an ultrametric the Pearson--Bellissard Laplacian is expected to be related to the motion of the transverse degrees of freedom of the aperiodic medium described by the tiling. But a detailed understanding of this expectation has still to be found.

We then review the results of a Dirichlet form defined for full substitution spectral triple of a
Pisot substitution tiling (Section~\ref{ssec-TilLap}). This time integration is defined using the spectral state. The main result is quite different from the Pearson--Bellissard Laplacian. In fact, the Laplacian for the  full substitution spectral triple can be interpreted as an elliptic differential operator with constant coefficients defined on the dual group of the group of topological eigenvalues of the dynamical system defined by the substitution tiling.

Section~\ref{sect-charorder} is devoted to a somewhat unexpected development which has to do with the fact that the spectral triple of a subshift is parametrised by a choice function. As a consequence we have a whole family of distance functions depending on that parameter. One may ask whether these are all equivalent in the sense of Lipschitz continuity. We consider hence the extremal values of the distance function, $\underline{d}$ and $\overline{d}$ and ask when there exists $c>0$ such that $c^{-1}\underline{d} \leq \overline{d} \leq c\underline{d}$. This turns out to be equivalent to a combinatorial property of the subshift. Such a constant $c>0$ exists if and only if the subshift has bounded powers, i.e.\ there exists $N>0$ such that no sequence ever contains an $N$-fold repetition of a word. 
For a Sturmian subshift of slope $\theta$ this combinatorial property is equivalent to the number theoretic property that $\theta$ has a bounded continued fraction expansion. This is an example in which noncommutative geometry can say something  about a certain combinatorial property of the subshift. This combinatorial property in turn can be understood as a criterion for high aperiodic order of the subshift. 

The final section is about the $K$-homology of compact ultrametric spaces.
Any spectral triple defines a $K$-homology class and therefore via Connes pairing
a group homomorphisms on $K$-theory with values in $\Z$. We consider these group homomorphisms in the context of the spectral triples we construct for compact ultrametric spaces. 
The flexibility of our construction allows us to design for every such homomorphism a spectral triple 
which defines it.

\section{Spectral triples}\label{sec-spectral-triples}
\subsection{General definition}

\begin{definition}\label{def-spectral-triple}
A spectral triple $(A,D,\HS)$ is given by a complex involutive
associative unital algebra $A$ which is faithfully represented by bounded operators
on some Hilbert space $\HS$ together with a self-adjoint operator $D$ on $\HS$
of compact resolvent such that all commutators $[D,a]$, $a\in A$, extend to
bounded operators.
The spectral triple is called {\em even} if there exists a grading operator on the Hilbert space such that $A$ is represented by even operators and $D$ is an odd operator.
\end{definition}
The basic idea is
that $[D,a]$ is the derivative of $a$ thus furnishing a differential
calculus on $A$, although it is a priori
not required that $[D,a]$ lies in the image of the representation (and hence
defines an element of $A$). 

Given a $C^*$-algebra $A$, a spectral triple for $A$ is is a spectral
triple $(A_0,D,\HS)$ in the above sense where $A_0$ is a dense
subalgebra of $A$ (we suppose that the representation of $A_0$ on
$\HS$ is continuous and hence also $A$ is represented on $\HS$).
If it is useful, we emphasize the representation $\pi$.  


There are a couple of additional requirements made usually to ensure
good properties and rigidify the theory. These are most often
motivated by Riemannian geometry. We will mention in our
discussion some of those, keeping an open mind not to  constrain to
much as our application in mind is to tilings.

We will now quickly browse through the main objects which can be defined by
means of  a spectral triple and which are or relevance for us.

\paragraph{Connes distance}  
The formula
\begin{equation}\label{eq-Connes-dist}
d_C(\sigma,\sigma') = \sup\{ |\sigma(a)-\sigma'(a)| : 
a\in A, \|[D,a]\|\leq 1\}
\end{equation}
defines a pseudo-metric on the state space $S(A)$ of $A$. This pseudo-metric
$d_C$ is a metric whenever the representation $\pi$ is non-degenerate and
$A_D' := \{a\in A:[D,a] = 0\}$ is one-dimensional, that is, contains only
multiples of the identity.

The state space comes with a natural topology, namely the weak-$*$ topology.
One may ask, when is the metric compatible with the weak-* topology?
Rieffel has provided a complete answer to this question~\cite{Ri04}: assuming
that $d_C$ is a metric, it generates the weak-* topology on $S(A)$ if and only
if the image of $B_1:=\{a\in A:\|[D,a]\| \leq 1\}$ in the quotient space
$A/A'_D$  is pre-compact. While complete, this characterisation is not always easy to
verify and we will indeed use more direct methods to verify whether
$d_C$ generates the topology or not.

This is already very interesting if $A$ is commutative. In this case
$A$ is isomorphic to $ C(X)$, for some  compact topological space $X$ which
is homeomorphic to
the closed subset of pure states on $A$. Eq.~\eqref{eq-Connes-dist} restricted to the pure states then becomes
$$
d_C(x,x') = \sup\{ |f(x)-f(x')| : 
f \in C(X), \|[D,f]\|\leq 1\}$$ and under Rieffel's conditions $d_C$
generates the topology of $X$. It is
therefore quite natural to require that Rieffel's conditions are satisfied.
On the other hand, however, if the construction of the spectral triple 
follows a natural principle one can use the criterion of whether or not the
Connes distance generates the topology as a characterisation of the
algebra $A$, or the space $X$; this is the basis of the characterisation of order
we exhibit in Section~\ref{sect-charorder}.

\paragraph{Zeta function}
Since the resolvent of $D$ is supposed compact $\TR(|D|^{-s})$ 
can be expressed as a Dirichlet series in terms of the
eigenvalues of $|D|$.\footnote{For simplicity we suppose (as will be the case
in our applications) that $ \ker(D)$ is trivial, otherwise we would
have to work with $\TR_{ \ker(D)^\perp }(|D|^{-s})$ or remove the
kernel of $D$ by adding a finite rank perturbation.}
The spectral triple is called {\em finitely summable} if the Dirichlet series 
is summable for some $s\in\RM$ and hence defines a function
\[
\zeta(z) = \TR(|D|^{-z})\,,
\]
on some half plane $\{z\in\CM: \Re(z)>s_0\}$ which is called the {\em
zeta-function} of the spectral triple. Under the right circumstances, $\zeta$
admits a meromorphic extension to the whole complex plane and then its pole
structure yields interesting information.
We will see that in particular the smallest possible value for $s_0$ 
in the above (the {\em abscissa of convergence} of the Dirichlet
series) that is, the largest pole on the real axis, is related to the
complexity of the tiling. 
This number $s_0$ is also called the {\em metric dimension} of the spectral
triple.

\paragraph{Spectral state and measure}
Under the right circumstances, the limit 
 \[
 \Tt(T) = \lim_{s\to s_{0}^+} \frac1{\zeta(s)} \TR(|D|^{-s}T).
 \]
exists for a suitable class of operators on $\Hh$. It then defines a
positive linear functional on this class of operators which we call
the {\em spectral state} defined by $(A,D,\Hh)$. It is of particular
interest already if $A=C(X)$ is commutative. Then the restriction of
$\Tt$ to $A$ defines by the Riesz representation theorem a measure on
$X$. We call that measure the {\em spectral measure} defined by $(A,D,\Hh)$.

\paragraph{Dirichlet forms}
Expressions of the type $(a,b)\mapsto \Tt([D,\pi(a)]^*[D,\pi(b)]$,
where $\Tt$ is a state, define quadratic forms on suitable subspaces
of $A$ which, under the right
circumstances, can be extended to Dirichlet forms on the $L^2$-space over
$A$ w.r.t.\ to the spectral state. In this definition $\Tt$ is often,
but not always also the spectral state. Indeed, we will see that the
choice $\Tt = \TR$ yields interesting Dirichlet forms. 
 
The interest in Dirichlet forms comes from the fact that they define Markov 
processes with generators which in the classical cases correspond to
Laplace--Beltrami operators. Under the right circumstances, a spectral
triple provides us therefore with a Laplacian.

\paragraph{$K$-homology}
The spectral triple yields directly an unbounded Fredholm module and
therefore the representative of a $K$-homology class of $A$. By the
Connes pairing of $K$-homology with $K$-theory the spectral triple
defines thus a functional on $K_*(A)$.


The points mentioned above are also interesting for commutative algebras.
We will in particular consider here the cases $A = C(\Xi)$ or $A=C(\Omega)$
where $\Xi$ is the discrete and $\Omega$ the continuous hull of a tiling.

\subsection{Spectral triples for metric spaces}

Let $X$ be a compact topological Hausdorff space. 
There exist various constructions for spectral triples for $X$, \emph{i.e.},
 for the algebra $C(X)$. 
These constructions are designed to fulfill additional properties. 
For instance, if $X$ is a Riemannian spin manifold with metric $g$
then one can use the Hilbert space of $L^2$-spinors with its standard 
Dirac operator defined with the help of the spin structure
to obtain a spectral triple which has the property that the Connes'
metric is equal to the one determined by $g$.
We won't describe this construction here, as it draws heavily on differential
geometry and we are interested in spaces which are far from being a manifold.
We would like to discuss two properties that spectral triples can satisfy--%
though not necessarily at the same time.
One is that the associated Connes' metric is equivalent 
to the original metric with a Lipschitz constant which is arbitrarily
close to~$1$, and the other is that 
the metric dimension coincides with the Hausdorff dimension or with
the lower box counting dimension of $X$.

The simplest case of an even spectral triple seems to us the spectral
triple of a pair of points. $X$ consists of two points so the algebra
is $C(X) =\C^2$ acting diagonally on $\HS=\C^2$ and the Dirac operator
is $D =\left( \begin{array}{cc} 0 & 1\\ 1 & 0 \end{array}\right)$. 
Its associated Connes distance gives the two points distance
$1$. Multiplying $D$ by $\frac1d$ the two points get distance
$d$. This spectral triple is the base for all what follows. The idea
is to approximate the metric space by finite point sets and to encode
that two points are considered to be neighbours with the help of
(horizontal) edges linking the two points. The spectral
triple of the approximation will then be a (finite) direct sum of pair
spectral triples.
Next, the approximation is refined, \emph{i.e.}\ $X$ is approximated by more points
which eventually become dense.
Taking a sequence of finer and finer approximations we will end up with
spectral triple for the space which will be a countable direct sum of pair 
spectral triples.
This idea occurs in the work of Christensen and Ivan~\cite{CI07}.
In the next section we will make this idea precise.

\subsubsection{Rooted trees and approximating graphs}
\label{sec-rtag}
Rooted trees are the first (and most fundamental) ingredient of our
construction of an 
approximating graph and its associated spectral triple. 
A rooted tree is a connected graph $\Tt$ without cycles 
(as an un-oriented graph) which has a distinguished vertex called the root. 
We denote by $\Tv$ the vertices and by $\Te$  the edges of the tree.
The edges of the graph may be oriented in
such a way that they point away from the root. By a path on the tree
we then mean a sequence of edges such that the endpoint (or range) of
the $n$th edge corresponds to the startpoint (or source) of the
$n\!+\!1$th edge. Here the start and endpoint are defined
w.r.t.\ the orientation. We write $v\preceq v'$ if there is a path (possibly of zero
length) from $v$ to $v'$.
Any vertex is the endpoint of a unique path which starts at the root vertex. 
We denote by  $\Tv_n$ the vertices whose corresponding unique path has length $n$ and
call $n$ also the level of the vertex. 
In particular $\Tv_0$ contains only the root vertex. 
Any vertex besides the root has one incoming edge.  
A branching vertex is a vertex which has at least two outgoing edges. We assume 
that it has finitely many outgoing edges and hence that each $\Tv_n$ is finite. 
For a vertex $v\in \Tv_n$ we let $\Tv(v):= \{v'\in\Tv_{n+1} : v\preceq v'\}$ be the set of 
its successors.
 
The boundary of the tree, denoted $\partial\Tt$ is defined as the set of
infinite paths  on $\Tt$ which start at the root vertex. We equip it
with the topology whose basis is in one-to-one correspondence with the
vertices of $\Tt$: each vertex $v$ defines a set $[v]$ of that basis, namely
$[v]$ is the set of all
infinite paths on $\Tt$ which pass through $v$. If $v\in \Tv_n$ 
then the complement of $[v]$ is the union of all $[w]$ where $v\neq
w\in \Tv_n$. Hence $\partial\Tt$ is totally disconnected space (it has
a basis of clopen sets). Moreover it is metrizable. Let $\xi,\eta$ be
two infinite paths on $\Tt$ which start at root. If they are distinct
then they will branch at a certain vertex, or stated differently,
they agree on a finite path from root on. Let $\xi\wedge\eta $ be that
finite path and $|\xi\wedge\eta|$ its (combinatorial) length. Then,
given any strictly decreasing function $\delta:\R^+\to\R^+$ which
tends to $0$ at $+\infty$,
$$ d(\xi,\eta) = \delta(|\xi\wedge\eta|) $$
defines a metric which induces the topology above. $d$ is in fact an
ultrametric, that is, it satisfies $d(x,y) \leq \max\{d(x,z),d(y,z)\}$
for all $x,y,z\in X$. Our assumption that $\Tv_n$ is finite implies
that $\partial\Tt$ is compact.

The data of an approximating graph for a compact space $X$ are the following.
\begin{enumerate}
  \item A rooted tree $\Tt = (\Tv,\Te)$.
  We assume that $\partial\Tt$ is a dense $1 : 1$ extension of $X$, that is,
  there is a continuous surjection $q: \partial\Tt  \to X$ such that the points
  which have a unique preimage form a dense set. 

  \item A non-empty symmetric subset 
    $$\HE\subset  \{(v,v')\in \Tv\times \Tv :
    v \not\preceq v',v' \not\preceq v\}$$
  whose elements we call horizontal edges
 interpreting them as the edges of a graph $(\Vv,\HE)$.
  Here $\Vv\subset\Tv$ is the subset of vertices which are the source (and hence
  also the range) of at least one element in $\HE$.
  This graph, which we call the horizontal graph, has no loop edges and no
  multiple edges and for each edge $(v,v')\in \HE$ contains also the edge with opposite orientation $(v',v)$. 
   We suppose that $(\Vv,\HE)$ is locally finite in the sense that each vertex
  $v \in \Vv$ is connected only to finitely many incoming or outgoing edges.


  \item There is a so-called length function $\delta:\HE\to \R^{>0}$ which
  satisfies $\delta(v',v) = \delta(v,v')$ and two further conditions:
(i) for all $\epsilon>0$ the set
$\{h\in\HE:\delta(h)>\epsilon\}$ is finite
(ii) there is a strictly decreasing sequence $(\delta_n)_n$ tending to
$0$ at $\infty$ such that $\delta(v,v')\geq \delta_{|v\wedge v'|}$.


  \item There is a so-called choice function $\tau:\Tv\to \partial\Tt$ which
  satisfies 
  \begin{enumerate}
    \item $\tau(v)$ goes through $v$. 
    \item if $w\prec v$ then 
    $\tau(w) = \tau(v)$ if and only if $\tau(w)$ passes through the vertex $v$.
    \item $q$ is injective on the image of $\tau$.
  \end{enumerate}
\end{enumerate}

Note that we have the following sequences of maps 
\begin{equation}\label{eq-maps} 
 \HE\xymatrix{\ar@<.3ex>[r]^{r}  \ar@<-.3ex>[r]_{s} &}
 \Vv \stackrel{\tau}{\to}  \partial\Tt  \stackrel{q}{\to} X
\end{equation}
where $r(v,v') = v'$ and $s(v,v')=v$.
Since $q\circ\tau\times q\circ\tau :\Hh\to X\times X$ is injective
the length function defines a distance function on $q(\tau(\Vv))$.

\begin{definition}
The approximating graph of the above data (1)--(4) is the (metric)
graph $G_\tau$ with vertices $ V =  q(\tau(\Vv))$ and edges 
$E = 
\{\big(q\circ\tau(v),q\circ\tau(v')\big):
(v,v')\in 
\HE\}$
equipped with the length function
$\big(q\circ\tau(v),q\circ\tau(v')\big) \mapsto \delta(v,v')$.
\end{definition}
For each pair of vertices $v,v'$ which are linked by an edge in $\HE$ we choose (arbitrarily)
an order calling the edge defined by the vertices in that order positively oriented and the one with the reversed order negatively oriented. This splits $\HE$ into two parts $\HE^\pm$ according to the chosen orientation of the edges. The splitting is, of course, non-canonical but the choice of orientation will not affect the final results. 
Since $E$ is in bijective correspondence this
orientation carries over to $E$. We denote by $\cdot^\op$ the
operation of changing the orientation of an edge.
Conditions (b) and (c) above imply that   $q\circ\tau(v)\neq q\circ\tau(v')$
if $(v,v')\in \HE$.
Let
\begin{equation}\label{equation-dense-subalgebra}
 C(\partial\Tt)_0 =  \left\{ f \in C(\partial\Tt): \sup_{(v,v')\in
     \HE} \frac{|f(\tau(v)) - 
  f(\tau(v'))|}{\delta(v,v')} < \infty
\right\}.
\end{equation}
\begin{lemma}\label{lem-dense}
$C(\partial\Tt)_0$ is a dense subalgebra of $C(\partial\Tt)$.
\end{lemma}
\begin{proof}
Let $\chi_v$ denote the characteristic function on $[v]$. Since $[v]$
is clopen $\chi_v\in C(\partial\Tt)$.
We claim that the expression $\chi_v(\tau(v_2)) -\chi_v(\tau(v_1))$
can be non-zero only if $v_1\wedge v_2 \prec v$. Indeed, if this is
not the case then either $v_1\wedge v_2 \succeq v$ or $v_1\wedge
v_2\wedge v \prec v,v_1\wedge v_2$. In the first case $\tau(v_1)$ and
$\tau(v_2)$ both contain $v$  and in the second both do not.
Since $v_1\wedge v_2 \prec v$ implies $\delta(v_1, v_2) \geq \delta_{|v|}$
we have $$\frac{|\chi_v(\tau(v_2))
  -\chi_v(\tau(v_1))|}{\delta(v_1,v_2)} 
\leq \frac2{\delta_{|v|}}$$
and so we see that $\chi_v\in C(\partial\Tt)_0$. Moreover, since $\partial\Tt$ is totally
disconnected the algebra generated by characteristic functions on $[v]$ is
dense in $C(\partial\Tt)$.
\end{proof}
\subsubsection{The spectral triple of an approximating graph}
\label{sec-ag}

We first consider a spectral triple over $\partial \Tt$. It depends on
the above data except the surjection $q$.
\begin{theorem}
\label{theorem-STbdry}
Consider a rooted tree $\Tt$ with a set of horizontal edges $\HE$, a
length function $\delta$, and a choice function $\tau$ as above. 
Let  $\HS = \ell^2(\HE)$ and represent $C(\partial\Tt)$ on $\HS$ by $\pi_\tau$,
$$ \pi_\tau(f)\psi(h) = f(\tau(s(h))) \psi(h).$$
Let the Dirac operator be given by 
$$ D\psi(h) = \frac1{\delta(h)}  \psi(h^{op}).$$
Then
$(C(\partial\Tt),D,(\HS,\pi_\tau))$ defines an even spectral triple w.r.t.\ the
decomposition $\HS^\pm =   \ell^2(\HE^\pm)$ defined by the orientation
of the edges. 
\end{theorem}
\begin{proof}
Since $|D|^{-1}\psi(h) = \delta(h)\psi(h)$ we see that $D$ has compact
resolvent if and only if for all $\epsilon>0$ the set
$\{h\in\HE:\delta(h)>\epsilon\}$ is finite. 

Furthermore
$$[D,\pi_\tau(f)]\psi(v_1,v_2) = \frac1{\delta(v_1,v_2)}\big(
f(\tau(v_2)) - f(\tau(v_1))\big) \psi(v_2,v_1).$$
Hence $C(\partial\Tt)_0$ is precisely the subalgebra of functions $f$
for which 
the commutator $[D,\pi_\tau(f)]$ is bounded.
\end{proof}
Using the surjection $q$ we can define a spectral triple on $X$.
\begin{theorem}
\label{theorem-STapgraph}
Consider an approximating graph $G_\tau$ as above
with length function $\delta$. Let  $\HS = \ell^2(E)$ and represent
$C(X)$  on $\HS$ by $\pi$,
$$ \pi(f)\psi(e) = f(s(e)) \psi(e).$$
Let the Dirac operator be given by 
$$ D\psi(e) = \frac1{\delta(e)}  \psi(e^{op}).$$
If ${q^*}^{-1} (C(\partial\Tt)_0)\subset C(X)$ is dense in $C(X)$ then
$(C(X),D,\HS)$ defines an even spectral triple w.r.t.\ the
decomposition $\HS^\pm =   \ell^2(E^\pm)$ defined by the orientation
of the edges. 
\end{theorem}
\begin{proof}
The only issue is the question whether 
$ C(X)_0 =\{ f\in C(X): \|[D,\pi(f)]\|<\infty\}$ is dense in $C(X)$. 
By construction ${q^*}^{-1} (C(\partial\Tt)_0)\subset C(X)_0$ and so our
hypothesis guarantees that property.
\end{proof}
We call this spectral triple the {\em spectral triple of the
  approximation graph $(G_\tau,\delta)$}. 
Note that if $q$ is a bijection and hence $X$ homeomorphic to
$\partial\Tt$ then the two spectral triples are the same. In the first
formulation the dependence on the choice function shows up in the
definition of the representation whereas in the second formulation it
is the embedding of the graph in $X$ which depends on it.

Without any further conditions on $\delta$ nothing can be said about
whether the Connes distance induces the topology of $X$. 

How do such triples arise? We will discuss in the next sections examples
with canonical tree structure, one being the case of a compact
ultrametric space and the other being the case of a substitution
tiling space. In these cases, more or less natural choices for the
horizontal edges $\HE$ and the length function $\delta$ can be
argued for. This is not the case for the choice function $\tau$
which therefore has to be regarded as a parameter. 
In~\cite{PB09} it is interpreted as the analogue of a tangent vector of a
manifold. 
\bigskip

An arbritrary compact metric space $(X,d)$ can be described by an approximation graph as above, although the involved graph is not canonical.
Palmer~\cite{Palmer} starts with a sequence  $(\Uu_n)_n$ of open covers of $X$.  
We call the sequence  {\em refined} if $\Uu_{n+1}$ is finer than $\Uu_n$ for all $n$, and {\em resolving} if it is refined and $\diam  \Uu_n\stackrel{n\to\infty}{\longrightarrow} 0$.\footnote{The diameter of a covering $\Uu$, written $\diam \Uu$, is the supremum of the diameters of the sets of $\Uu$.}
A resolving sequence {\em separates points}: for any $x,y\in X$, $x\neq y$, there exists $n$ and $U,U'\in\Uu_n$ such that $x\in U$, $y\in U'$, and $U\cap U'=\emptyset$.
In general such a sequence $(\Uu_n)_n$ is by no means unique.
There is a graph $\Tt=(\Tv,\Te)$ associated with such a sequence: vertices in $\Tv_n$ are in one-to-one correspondance with open sets in $\Uu_n$; this correspondance is written $v\leftrightarrow U_v$.
The inclusion $U_v\subset U_w$ for $U_v\in\Uu_{n}$, $U_w\in \Uu_{n-1}$ defines an edge of $\Te_n$ with source $w\in\Tv_{n-1}$ and range $v\in\Tv_n$. 
An infinite path $\xi\in\partial\Tt$ stands for a sequence of open
sets $(U_{\xi_n})_n$ such that  $U_{\xi_n}\subset U_{\xi_{n+1}}$.
As the sequence $(\Uu_n)_n$ is resolving, the intersection $\bigcap_n U_{\xi_n}$ contains a single point $q(\xi)$.
This defines a map $q: \partial\Tt \rightarrow X$, which turns out to be a continuous surjection. 
One has now to introduce choice and length functions to build an approximation graph as above and obtain the spectral triple of Theorem~\ref{theorem-STapgraph}.

Palmer considers choice functions which are slightly more general than
here (sets of horizontal edges $\Hh$ can be deduced from his choice
functions). 
The length function $\delta$ is simply the actual distance of the
points in $X$. These choices fix an approximation graph $(G_\tau,\delta)$.
Palmer's main concern is to show the existence of a resolving sequence
$(\Uu_n)_n$ such that, for any choice function, he gets a spectral
triple recovering the Hausdorff dimension of $X$ and the Hausdorff
measure on Borel sets of $X$. 
\begin{theorem}[\cite{Palmer}]\label{thm-Palmer}
There exist a resolving sequence $(\Uu_n)_n$ such that, for any choice
function, the spectral triple for $X$ defined by the approximation
graph  constructed from the above data has
\begin{enumerate}
 \item metric dimension $s_0$ equal to the Hausdorff dimension of $X$, and
 \item spectral measure equal to the (normalized) Hausdorff measure on $X$.
\end{enumerate}
\end{theorem}
\noindent As a corollary, the Hausdorff measure of $X$ can be shown to be $\displaystyle{ \lim_{n\to \infty} \sum_{U\in\Uu_n} (\diam U)^{s_0}}$.

\bigskip

Finally we mention that the approach of Christensen \& Ivan
\cite{CI07} (which predates~\cite{Palmer}) can be recast as well in
the framework of approximating graphs. One does not have to work with
an abstract tree and choice functions to construct such a graph but
may simply proceed as follows: Start with a sequence $(V_n)_n$ of
finite subsets $V_n$ of $X$ such that $V_n\subset V_{n+1}$ and their
union $\bigcup_n V_n$ is dense in $X$.
 For each $n$ choose a non-empty
symmetric subset $E_n\subset V_n\times V_n$ and define $\delta(x,y) =
d(x,y)$, \emph{i.e.}\ the length of edge $(x,y)$ corresponds to their
distance in $X$. Now one has directly the approximating graph $(V,E)$
with metric. But there is just enough structure to recover the metric
aspects, namely one gets:
\begin{theorem}[\cite{CI07}]
\label{theorem-Christensen} Let $(X,d)$ be a compact metric space.
For any constant $\alpha>1$ there exists a sequence $(V_n)_n$ of finite point sets
$V_n\subset X$ together with a choice of horizontal pairs $E_n$
as above, such that the spectral triple of
the approximating graph $(V,E)$ with length function $\delta$ yields a spectral metric $d_C$
which satisfies $$ d(x,y)\leq d_C(x,y) \leq \alpha d(x,y).$$
\end{theorem}

\subsection{Spectral triples for compact ultrametric spaces}\label{sec-ultra}

In this section we consider spectral triples for compact ultrametric spaces.
With a particular choice for the horizontal egdes, we obtain the spectral triples of
Pearson--Bellissard~\cite{PB09}. The general case has been discussed
in~\cite{KS11}.

A compact ultrametric space $(X,d)$ is a metric space for which the metric
satisfies the strong triangle inequality
\begin{equation}\label{ultra} 
d(x,y) \leq \max\{d(x,z),d(y,z)\} 
\end{equation}
for all $x,y,z\in X$. We suppose that $X$ has infinitely many points (the finite case being simpler).
Such spaces arise as subshifts or as transversals of spaces (``discrete tiling spaces'')
of non-periodic tilings with finite local complexity.
The property (\ref{ultra}) implies that the open $\delta$-balls
$B_\delta(x)=\{y\in X:d(x,y)<\delta\}$ satisfy
either $B_\delta(x)=B_\delta(y)$ or
 $B_\delta(x)\cap B_\delta(y)=\emptyset$. 
In particular there is a unique cover by $\delta$-balls
($\delta$-cover). Moreover this cover is a partition and hence $X$ is
totally disconnected. Furthermore, the image of $d$ contains exactly
one accumulation point, namely $0$. 
In other words, there exists a strictly decreasing sequence 
$(\delta_n)_n$ converging to zero such that 
$\mbox{\rm im}\, d = \{\delta_n :n\in\N\}$. 
If we take $\Uu_n$ to be the $\delta_n$-cover of $X$
we obtain a canonical refined resolving sequence of coverings. Its
associated tree $\Tt$ is the so-called Michon tree
\cite{Mich85,PB09,Palmer}. In particular, the
vertices of level $n$ of $\Tt$ correspond to the clopen
$\delta_n$-balls of the $\delta_n$-cover, and the root corresponds to all of $X$.
The sets of the covering are also closed, so given an infinite path $\xi$,
the sequence of vertices through which it passes defines a nested sequence
of compact sets whose radius tends to~$0$. It defines a point $x(\xi) \in X$.
The map $\xi\mapsto x(\xi)$ is even injective in the case of ultrametric spaces
and thus furnishes a homeomorphism between $\partial\Tt$ and $X$.  

Now that we have the canonical tree associated with the compact
ultrametric space we need to choose horizontal edges. Any vertex has
one incoming vertical edge.  
A branching vertex is a vertex which has at least two outgoing
(vertical) edges. 
For a vertex $v\in \Tv_n$ let $\Tv(v):= \{v'\in\Tv_{n+1} : v\preceq v'\}$.   
A canonical choice for the horizontal edges $\HE$
is to introduce an edge
between any pair of distinct vertices of $\Tv(v)$, and this for any
branching vertex. This is the {\em maximal} choice. 
A {\em minimal} choice would be to choose, for each branching vertex
$v$, two distinct vertices of $\Tv(v)$ and to introduce two horizontal
edges, one for each direction,  only for these two. 
This case has been considered in~\cite{PB09}. 
Note that in both choices we obtain a grading for the horizontal edges:
In the maximal choice $\HE^{\max} =\bigcup_n\HE_n^{\max}$ with 
$$ \HE_n^{\max} = \{(v',v'')\in \Tv(v), v\in\Tv_{n-1}, v'\neq v''\} $$
and a minimal choice is a subset $\HE^{\min} =\bigcup_n\HE_n^{\min}$,
with $\HE_n^{\min}\subset  \HE_n^{\max}$.
More generally we may consider horizontal
edges of the form 
\begin{equation}\label{eq-HE}
\HE =\bigcup_n\HE_n, \qquad \HE_n\subset  \HE_n^{\max}.
\end{equation}

The natural length function for all these cases is the
function determined by the radii of the balls, \emph{i.e.}\ 
$\delta(v,v') = \delta_n$, where $(v,v')\in\HE_n$. 

Any choice function $\tau$ provides us with a spectral triple for the
compact ultrametric space $X$ according to
Theorems~\ref{theorem-STbdry} or \ref{theorem-STapgraph} (since $q$ is a homeomorphism  the two theorems are equivalent in the present case).
\begin{definition}
\label{def-stultra}
By a spectral triple for a compact ultrametric space we mean the
spectral triple given by the above data: its Michon tree, a choice of horizontal edges
$\HE$ as in (\ref{eq-HE}), 
length function determined by the radii of the balls, and any choice
function.
If $\HE=\HE^{\min}$ we refer to the spectral triple also as
Pearson--Bellissard spectral triple.
\end{definition}
%

\begin{lemma}
Consider a spectral  triple for a compact ultrametric space.
Its spectral distance $d_C$ bounds the original metric
$d(x,y)\leq d_C(x,y)$.
If, for any $v',v''\in\Tv(v)$ (with $v\in\Tv$), there is a path of edges in
$\HE$ linking $v$ with $v'$ then $G_\tau$ is connected and hence $d_C$ a metric.
\end{lemma}

The condition on $\HE$ formulated in the lemma is obviously
satisfied for the maximal choice $\HE^{\max}$. 

\section{Concrete examples for subshifts and tilings}
\label{sec-concrete}

We discuss applications of the previous constructions to subshifts and
to tilings of finite local complexity.  
\begin{exam}
\label{ex-Fibo}
We will illustrate some of our constructions using the example of the Fibonacci tiling.
We will consider three versions of it:
\begin{enumerate}

\item The Fibonacci one-sided subshift;

\item The Fibonacci two-sided subshift;

\item The one-dimensional Fibonacci substitution tiling.

\end{enumerate}
The  Fibonacci subshift (one-sided or two-sided) is a subshift of the full (one-sided or two-sided) shift over the alphabet $\Aa =\{ a, b\}$ defined usually by the substitution $\sigma$,
$\sigma(a) = ab$, $\sigma(b)= a$. 
This means by definition that the elements of the subshift are infinite sequences of letters, that is, functions $\N\to\Aa$ (one-sided) or functions $\Z\to\Aa$ (two-sided), 
whose finite parts (words) are allowed in the sense that they occur as subwords
of $\sigma^n(a)$ for some $n$ (which depends on the word). Note that instead of $\sigma$ we could also take $\sigma^2$ or even the substitution $a \mapsto baa$, $b\mapsto ba$ to define the same subshift, because all these substitutions yield the same notion of allowed words. The advantage of the latter substitution is that it manifestly forces its border. 
 
 The one-dimensional Fibonacci substitution tiling is the suspension of the two-sided subshift,
 in which the letters are realised as intervals (one-dimensional tiles),
$a$ by an interval of length $(\sqrt{5}+1)/2$ and $b$ by one of length $1$.
Alternatively this is a canonical cut-and-project tiling of $\ZM^2 \subset \RM^2$ onto $\RM^1$ as the line throughout the origin with irrational slope $(\sqrt{5}+1)/2$, and window the half-open unit cube $[0,1)\times [0,1)$.
\end{exam}

\subsection{One-sided subshifts and the tree of words}
\label{ssec-STsubshift}
Our first application is to the space of one-sided sequences of a
subshift. The construction, which is described in \cite{KS11,KLS11},  
is based on the so-called tree of words of the subshift.

Recall that a one-sided full shift over a finite alphabet $\Aa$ is a the set of sequences $\Aa^\N$
considered as a compact topological space whose topology is that of pointwise convergence.
On this space we have an action of $\N$ by left shift, that is, dropping the first letter.
A subshift is a closed, shift-invariant subset of the full shift and its language $\Ll$ is the set of finite words occurring in the sequences of the subshift.  

The tree of words for a one-sided subshift with language $\Ll$ is defined as
follows: the vertices of level $n$, noted $\Tv_n$, are the (allowed) words of
length~$n$ and the empty word corresponds to the root.
Given a word $w\in\Ll$ and any of its one-letter extensions $wa \in \Ll$,
$a \in \Aa$, we draw an edge from $w$ to $wa$.
Hence any word has exactly one incoming edge and at least one outgoing edge.
A word is called right-special if it can be extended to the right by 
one letter in more than one way. A right-special word corresponds thus
to a branching vertex in the tree. Note that if the subshift is
aperiodic then there is at least one right-special word per length.   

We will consider three interesting choices for the horizontal edges.
\begin{itemize}
\item {\em The  maximal choice} $\HE^{\max}$ as introduced in Section~\ref{sec-ultra}. With the maximal choice there is a  
horizontal edge between any two distinct one letter extensions of a right-special word.
The level of such an edge is thus equal to the length of the word plus~$1$.
\item {\em A minimal choice} as introduced in Section~\ref{sec-ultra}.
For each right-special word one chooses
  a pair of distinct one letter extensions which then are linked by
  two edges (one for each orientation). 
\item {\em The privileged choice} $\HE^{\mathrm{pr}}$. 
This is like the maximal choice but with respect to a certain subset of
so-called {\em privileged words}.
Concretely, there is a horizontal edge between any two distinct privileged
extensions of a privileged word.
The definition of privileged words is based on return words.
A word $v\in\Ll$ is a complete first return to a word $u\in\Ll$
if $u$ occurs exactly twice in $v$, namely once as a prefix and
once as a suffix.
By convention, a letter is a complete first return to the empty word.
A privileged word is defined iteratively: the empty word is a  privileged word,
and a complete first return to a privileged word is a privileged word.
Now $(v,v')\in\HE^{\mathrm{pr}}$ if and only if $v$ and $v'$ are privileged and there is
a privileged word $u\in\Tt^{(0)}$ such that $v$ and $v'$ are distinct complete
first returns to $u$. 
\end{itemize}
We will specify the length function later according to our needs. 
\begin{definition}
By a spectral triple for a one-sided subshift we mean the
spectral triple as defined in Theorems~\ref{theorem-STbdry} or \ref{theorem-STapgraph}
given by the above data: the tree of words, a choice of horizontal edges
$\HE$ as above, a choice of
length function $\delta$, and any choice
function.
\end{definition}\label{def-STsubshift}

We are interested in two different types of subshifts:
minimal aperiodic subshifts in which case the subshift (or its two-sided
version) stands for the symbolic version of a one dimensional tiling, or
subshifts of finite type, which arise in the context of substitution tilings.

 \paragraph{Example} Four levels of the tree of words $\Tt$ for the one-sided Fibonacci subshift of Example~\ref{ex-Fibo} are shown below.
There is a unique right-special factor per length, so each vertex of the tree has at most two successors, and therefore $\Hh=\Hh^{\rm min}=\Hh^{\rm max}$.
Letters stand for the vertices $\Tt^{(0)}$ of the tree, and vertical lines for its edges $\Tt^{(1)}$.
Bold letters stand for privileged words (this includes the root, {\it i.e.} the empty word).
Horizontal arrows (unoriented) $\Hh$ are represented by curvy lines, and privileged arrows $\Hh^{\rm pr}$ by dotted curvy lines.

\xymatrix{
& & &    & \mathbf{\emptyset}  \ar@{-}[dl] \ar@{-}[dr] & & &\\
& & & \mathbf{a} \ar@{-}[dl] \ar@{-}[d] \ar@{~}[rr] \ar@/_1pc/@{~~}[rr] & & \mathbf{b} \ar@{-}[dr] & & \\
& & \mathbf{aa} \ar@{-}[dl] \ar@{~}[r]  \ar@{~~}[dr] & ab \ar@{-}[d] & & & ba \ar@{-}[d] \ar@{-}[dr] & \\
& aab  \ar@{-}[dl] && \mathbf{aba}  \ar@{-}[dl] \ar@{-}[dr] & & & baa \ar@{-}[d] \ar@{~}[r] & \mathbf{bab} \ar@{-}[dr] \ar@{~~}[dl] \\
aaba & & abab \ar@{~}[rr] & & abaa && \mathbf{baab} && baba
}

\subsubsection{One-sided subshifts of finite type}
\label{ssec-finitetype}

Consider a finite oriented graph $\Gg=(\Gv,\Ge)$ such as
the graph associated with a substitution defined on an alphabet.
It defines a one-sided subshift of finite type: this is the subshift whose
alphabet is $\Gv$ and whose language is given by the finite paths on $\Gg$. 
The tree of words associated with the subshift of finite type looks as follows:
$\Tv_0$ contains the root vertex, $\Tv_1 = \Gv$, and 
$\Tv_{n}=\Pi_n(\Gg)$ is the set of paths of length $n$ on $\Gg$.
Furthermore, $\Te$ contains  one edge joining the root vertex to each
$v \in \Tv_1$. It contains also, for each path $\gamma$ over $\Gg$ and each edge
$\epsilon\in\Ge$ with $s(\epsilon)=r(\gamma)$,  one edge joining $\gamma$ to $\gamma \epsilon$ (we
denote the latter edge by $(\gamma,\gamma\epsilon)$).
It is clear that  $\partial\Tt=\Pi_\infty(\Gg)$, the set of infinite paths
over $\Gg$. 


When the tree of words is built from such a graph, it is possible to define a stationary Bratteli diagram, with a somewhat simpler description than the tree, such that the set of paths on
the diagram corresponds canonically to the set of paths on the tree. Such a stationary Bratteli diagram exhibits more clearly the underlying self-similarity than the tree. 
We will not give details for the construction based on Bratteli diagrams, but the construction suggests that it is natural to have a {\em self-similar} choice of horizontal edges,
in the sense that it should only depend on $\Gg$.
Choose a symmetric subset 
\[
\hat \HE \subseteq \left\{ (\varepsilon,\varepsilon') \in \Ge \times \Ge \ : \ 
  \varepsilon\neq \varepsilon', \;
  s(\varepsilon)=s(\varepsilon') \right\} 
\]
which we call {\em fundamental horizontal edges}
and then ``lift'' these to horizontal edges as follows:
\begin{equation}
\HE_n = \{( \gamma \epsilon,\gamma\epsilon') :
\gamma\in\Pi_n(\Gg), (\varepsilon,\varepsilon')\in\hat \HE , r(\gamma)
= s(\epsilon)\}
\end{equation} 
Note that vertices of $\Tv$ were by definitions paths on $\Gg$, so the
equation above defines indeed an horizontal edge as a pair of vertices on the
tree.
We fix an orientation on $\hat\HE$ and carry this orientation over to $\HE_n$.

We call the length function {\em self-similar} if there exists a
$0 < \rho < 1$ such that for all $h\in \HE_n$ 
\[\delta(h) = \rho^n.\]

The role of the choice function is to associate an infinite extension to each word.
This can also be understood in a way that for each word of length $n$ we make a choice of one-letter extension. In the context of subshifts of finite type it is natural to restrict the choice function
in the following way. Let us suppose that $\Gg$ is connected in the stronger sense that for any two vertices $v_1,v_2$ there exists a path from to $v_1$ to $v_2$ and a path from $v_2$ to
$v_1$, and that it contains a one-edge loop.\footnote{By going over to a power of the substitution
  matrix we can always arrange that the substitution graph has these
  properties if the substitution is primitive.}
We fix such a one-edge loop  $\epsilon^*$.
Consider a function $\hat\tau:\Ge\to\Ge$ satisfying that for all $\varepsilon\in \Ge$:
\begin{enumerate}
\item if $r(\varepsilon)$ is the vertex of $\epsilon^*$ then $\hat\tau(\varepsilon)=\epsilon^*$,
\item if $r(\varepsilon)$ is not the vertex of  $\epsilon^*$ then
  $\hat\tau(\varepsilon)$ is an edge starting at $r(\varepsilon)$ and
  such that $r(\hat\tau(\varepsilon))$ 
is closer to the vertex of $\epsilon^*$ in $\Gg$.  
\end{enumerate}
Then $\hat\tau$ defines an embedding of $\Pi_n(\Gg)$ into $\Pi_{n+1}(\Gg)$ by
$\varepsilon_1\cdots \varepsilon_n\mapsto \varepsilon_1\cdots
\varepsilon_n\hat\tau(\varepsilon_n)$ and hence, by iteration, into
$\Pi_\infty$.
The corresponding inclusion $\Pi_n(\Gg)\hookrightarrow \Pi_\infty(\Gg)$ is our
choice function $\tau$.
\begin{definition}\label{def-ST-self-similar}
With the above self similar choice of horizontal, length function, and
choice function we call the spectral triple of a subshift of
finite type {\em self-similar}.
\end{definition}

\paragraph{Example} The substitution graph $\Gg$ of the Fibonacci substitution subshift of Example~\ref{ex-Fibo} is shown below.
The arrows pointing towards the left vertex correspond to the occurrences of $a$ and $b$ in the substitution of $a$ ($baa$), while those pointing towards the right vertex correspond to the occurrences of $a$ and $b$ in the substitution of $b$ ($ba$)--the dot showing which letter is concerned.
\[
\xymatrix{
& a  \ar@(l,u)^{ba\dot{a}}  \ar@(l,d)_{b\dot{a}a} \ar@/^1pc/[rr]^{b\dot{a}} & & b \ar@/^1pc/[ll]^{\dot{b}aa} \ar@(ur,dr)^{\dot{b}a} 
}
\]
\noindent For $\hat\Hh$, we choose for example the pairs of edges $(\dot{b}a,b\dot{a}a)$ and $(\dot{b}a,\dot{b}aa)$ in $\Gg$.
For the edge loop we choose for instance $\epsilon^\ast = b\dot{a}a$.
We show a portion of the tree $\Tt$ below, together with horizontal edges $\Hh$ (lifting $\hat \Hh$) as curvy lines.

\xymatrix{
 & &    & \mathbf{\emptyset}  \ar@{-}[dl] \ar@{-}[dr] & & \\
 & & a \ar@{-}[dl] \ar@{-}[d] \ar@{-}[dr] \ar@{~}[rr]  & & b \ar@{-}[d] \ar@{-}[dr] &  \\
 & \dot{b}a \ar@{-}[dl]  \ar@{-}[d] \ar@{~}[r]  & b\dot{a}a \ar@3{-}[d] & ba\dot{a} \ar@3{-}[d] & \dot{b}a  \ar@{-}[d] \ar@{-}[dr]   \ar@{~}[r] & \dot{b}aa  \ar@3{-}[dr]   \\
b\dot{a}\,baa\ar@{~}[r]& b\dot{a}\,baa\,baa& \ldots & \ldots &  \dot{b}a\,baa \ar@{~}[r]& \dot{b}a\,baa\,baa & \ldots 
}

\subsection{The tree of patches and the ordinary transverse spectral triple of a tiling}
\label{ssec-STtrans} 
The construction of the tree of words for one-sided subshift can be
generalized to two-sided shifts, $\Z^d$-subshifts, and even tilings of
finite local complexity \cite{KS11}. This generalization is based on the
definition of an $r$-patch. The most common definition in the context
of tilings is to pick consistently a privileged point in each tile (puncture),
for example their barycenter.
The \emph{transversal} of the tiling space is then the set of all tiles
which have a puncture at the origin.
An $r$-patch is a patch which has a puncture on $0$ and just covers
$B_r(0)$, the Euclidean $r$-ball around the origin.\footnote{For $\Z^d$
  subshifts--which can be viewed as tilings by coloured cubes--it might be more
  handy to use cubes instead of Euclidean balls.} 
Recall that finite local
complexity means that, for any $r>0$, there are only finitely many
$r$-patches. The larger $r$ the more  $r$-patches there are. The
number of $r$-patches is a semi-continuous function of $r$ (the
so-called complexity function) and  
the points where this function jumps form an increasing sequence
$(r_n)_n$ of~$\R^+$.
Given an $r_n$-patch $v$, its diameter is noted $|v|:=r_n$.

The tree of patches $(\Tv,\Te)$ of a tiling of finite local complexity
is now constructed as follows: the level $n$ vertices are the $r_n$-patches and 
its root represents the empty patch ($r_0=0$).
There is a (vertical) edge between an $r_n$-patch and any of its extensions to
an $r_{n+1}$-patch and all edges arise in this way.

The tree of patches is the Michon tree of the transversal of the
tiling equipped with an ultrametric of the form 
\begin{equation}
\label{ttmetric}
d(\xi,\eta) = \inf\{\delta(r_n) \, : \, \xi_n=\eta_n, \text{ with } r_n = |\xi_n| \}\,,
\end{equation}
where $\delta:\R^+\to \R^+$ is any strictly decreasing function
converging to $0$ at $+\infty$.

No particular structure of the tiling seems to point to a natural
choice for the horizontal edges, the function $\delta$ above,
or for determining the length of edges, or the choice function $\tau$.
These data have to be chosen according to the specific situation in order to
define suitable approximating graph $G_\tau$ and $\delta$.
\begin{definition}\label{def-STtrans} 
By an {\em ordinary transverse spectral triple} for a tiling we mean a 
spectral triple for its canonical transversal (as in Def.~\ref{def-stultra})
equipped with an
ultrametric of the form \eqref{ttmetric}. 
\end{definition}
The spectral triple depends on  
a choice of horizontal edges $\HE$, a choice of strictly decreasing function 
$\delta:\R^+\to \R^+$ tending
to $0$, and a choice function.

\subsection{Substitution tilings}
\label{ssec-STsubst}
We consider now aperiodic primitive substitution tilings of finite local complexity. 
To these we may apply the construction of Section~\ref{ssec-STtrans}
to obtain an ordinary transverse spectral triple. But the extra
structure coming from the substitution map allows one also to consider
another spectral triple, namely the 
spectral triple of a one-sided subshift of finite type given by the
substitution graph (see Section~\ref{ssec-finitetype}). 
The advantage of this latter approach lies in the
fact that it can be extended into the longitudinal direction and therefore will
provides us with a spectral triple for the continuous hull.  
We follow \cite{KS13}.

Consider a finite set $\mathcal{A}=\{t_i:i\in\Gv\}$ of
translationally non-congruent tiles (called {\em prototiles}) in $\R^d$ indexed by a finite set $\Gv$.
A {\em substitution} $\Phi$ on $\mathcal{A}$ with expansion factor $\theta>1$ is a decomposition rule followed by and expansion by $\theta$, namely $\Phi$
assigns to each prototile a patch of tiles in $\R^d$
with the properties: for each $i\in\Gv$, every tile in $\Phi(t_i)$ is a translate of an element of $\mathcal{A}$; and the subset of $\R^d$ covered by $\Phi(t_i)$ is the subset covered by $t_i$ stretched by the factor $\theta$. Such a substitution naturally extends to patches and even tilings whose elements are translates of the prototiles and it satisfies
$\Phi(P-t) = \Phi(P)-\theta t$.

A patch $P$ is {\em allowed} for $\Phi$ if there is an $m\ge1$, an $i\in\{1,\ldots,k\}$, and a $v\in\R^N$, with $P\subset \Phi^m(t_i)-v$. The {\em substitution tiling space} associated with
$\Phi$ is the collection $\OP$ of all tilings $T$ of $\R^d$ such that every finite patch in $T$ is allowed for $\Phi$. $\OP$ is not empty and, since translation preserves allowed patches, $\R^d$ acts on it by translation. 

We assume that the substitution $\Phi$ is {\em primitive}, that is, for each pair $\{t_i,t_j\}$ of prototiles there is a $k\in\N$ so that a translate of $t_i$ occurs in $\Phi^k(t_j)$. We also assume that all tilings of $\OP$ are  {\em non-periodic} and have FLC. It then follows that $\OP$ is compact in the standard topology for tiling spaces and that $\OP=\Omega_T:=\overline{\{T-x:x\in\R^d\}}$ for any $T\in\OP$.
It also implies 
that the substitution map $\Phi$ seen as a map $\OP \rightarrow \OP$ is a homeomorphism (in particular it is bijective).

The substitution graph of $\Phi$ is the finite oriented graph $\Gg = (\Gv,\Ge)$ whose vertices $\Gv$ stand for the indices of prototiles and whose edges encode the position of tiles in supertiles. More precisely, given two prototiles $t_i$ and $t_j$ the supertile $\Phi(t_i)$ may contain several tiles 
of type $t_j$ (that is, tiles which are translationally congruent to $t_j$). These tiles are at different positions in the supertile $\Phi(t_i)$. 
For each possible position 
we introduce one oriented edge $\epsilon\in\Ge$ with range $r(\epsilon)=i$ and  source 
$s(\epsilon)=j$.

%
The canonical transversal of a substitution tiling (alternatively a substitution subshift) can be encoded by a one-sided shift of finite type.
To do so, one defines a map from the transversal $\Xi$ to itself by desubstitution: if $T \in \Xi$, the origin is a pointer inside of a tile $t_i$ of $T$.
Therefore, in $\Phi^{-1}(T)$, the origin is inside of a unique tile $t_j$, whose pointer $x_j$ doesn't need to be at the origin. Let $T':= \Phi^{-1}(T)-x_j \in \Xi$.
The map $T \mapsto T'$ is a continuous surjection $\Xi \rightarrow \Xi$.
For a given tiling $T$, the map $T \mapsto T'$ defines an edge in the substitution graph of $\Phi$ given by the position of $t_i$ in the supertile $\Phi(t_j)$.
By iteration, we get a continuous map $\Xi \rightarrow \Pi_\infty$ which commutes with the desubstitution map on $\Xi$ and the shift on $\Pi_\infty$ respectively.
This map is called the Robinson map~\cite{Grunbaum,Kel95} and denoted $\Rob$. Under an additional assumption on the substitution (called border-forcing), it is a homeomorphism.\footnote{Given a substitution tiling space $\Omega$, it is always possible to find a substitution which produces this space and forces its border. It involves decorating the prototiles (hence increasing their number).}
Given $t_i$ the tile at the origin of a tiling $T$ and $t_j$ the tile at the origin of $\Phi^{-n}(T)$, then $t_i$ is a tile included in the patch $\Phi^n(t_j)$, or we can say that $t_i$ is included in the $n$-supertile of type $t_j$.
Remark how each finite path of length $n$ corresponds to a tile at the origin (of type given by the source of the path) included in a $n$-supertile (given by the range of the path).
See Example~\ref{ex-Fibo} of the Fibonacci substitution at the end of Section~\ref{ssec-finitetype}.

The same construction can be done for a tiling which is not in the transversal: if the origin belongs to a unique tile in $\Phi^n (T)$ for all $n \in \Z$, one can define similarly an element $\Rob (T) \in \Pi_{-\infty,\infty}(\Gg)$,  where $\Pi_{-\infty,\infty}(\Gg)$ denotes the bi-infinite sequences over $\Gg$.
It is of course not defined on all of $\Omega$, but on a dense subset. It is injective, however, and its inverse can be extended to a continuous surjection $q:\Pi_{-\infty,\infty}(\Gg) \rightarrow \Omega$.
The interpretation is the following: given $\gamma \in \Pi_{-\infty,\infty}$, the indices $(\gamma_i)_{i \geq 0}$ define a tiling $T \in \Xi$ by the inverse of the Robinson map.
The indices $(\gamma_j)_{-N < j < 0}$ define a $N$-th order microtile of type $s(\gamma_{-N})$ inside of the tile $s(\gamma_0)$. As $N$ grows, it defines a decreasing sequence of microtiles which converge to a point $x$. Then $T-x$ is the tiling corresponding to $\gamma$.
It can happen that two sequences of microtiles converge to the same point, meaning $q$ is not injective.

We  may suppose\footnote{This can always be achieved by going over to
a power of the substitution.} that the substitution has a fixed
point $T^*$ such that the union over $n$ of the $n$-th order
supertiles of $T^*$ on $0$ covers $\RM^d$. 
Then $q^{-1}(\{T^*\})$ contains a single path and this path is constant, that
is, the infinite repetition of a loop edge which we choose to be $\epsilon^*$.
And we pick a choice function $\hat\tau$ on $\Gg$, as in Section~\ref{ssec-finitetype}, which induces an embedding \(\tau : \Pi(\Gg) \to \Pi_\infty(\Gg)\)
(if $\gamma$ is a finite path, then $\tau(\gamma)$ is an infinite path which eventually is an infinite repetition of the edge $\epsilon^\ast$).

Let 
$\Xi_{t}$ be the acceptance domain of prototile $t$ (the set of all tilings in $\Xi$ which have
$t$ at the origin).
Note that the sets $\Xi_{t_v}\times t_v$ (for $v\in\Gv$) cover $\Omega$. Their
intersection turns out to have measure~$0$.
Let $\Pi_{-\infty,\infty}^v$ be the set of bi-infinite paths which pass
through $v$ at level $0$.
Then $q_v:\Pi_{-\infty,\infty}^v\to \Xi_{t_v}\times t_v$ is a
continuous  almost one-to-one surjection.

In the section, we will use the chair substitution as an example.
It is primitive and aperiodic but does not force its border. The Robinson map can still be defined, but is
not a homeomorphism (\emph{i.e.} the one-sided shift of finite type is not an accurate
representation of the transversal).
For the sake of simplicity, we will nevertheless use it as an illustration.

\begin{figure}[htp]
\begin{center}
 \includegraphics{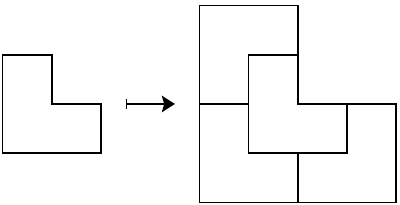} \qquad \qquad
 \includegraphics{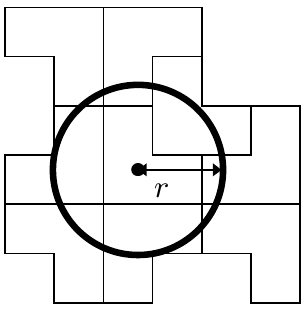} \qquad \qquad
 \includegraphics{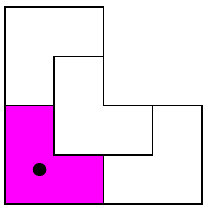}
 \caption{{\small On the left, the chair substitution rule (with four prototiles, one for each orientation). In the middle, a vertex of the usual tree of patch associated with radius $r$. On the right, a vertex of the self-similar tree of patch as described on this section: a tile (colored) inside of a $1$-st order supertile (equivalently, it is a length-one path in the substitution graph).}}
 \label{fig-chair-subst}
\end{center}
\end{figure}

\subsubsection{Transverse substitution spectral triple of a substitution tiling}\label{ssec-trans}
We use now the positive part
$q_v^+: \Pi^v_{0,+\infty}(\Gg)\to \Xi_{t_v}$ of $q_v$ to construct a second spectral triple for the transversal of a substitution tiling. Note that $q_v^+$ corresponds
to the inverse of a restriction of $\Rob$ and thus is a homeomorphism.

We choose fundamental (transversal) horizontal edges
\[
\hat\HE_\mathrm{tr}\subset \left\{ (\epsilon,\epsilon') \in \Ge\times\Ge \, :
  \, \epsilon\neq\epsilon', \; s(\epsilon)=s(\epsilon') \right\}\,  
\]
which we suppose to satisfy the condition 
\begin{itemize}
\item[(C)] if $s(\epsilon) = s(\epsilon')$ there is a path of edges in $\hat\HE_\mathrm{tr}$ linking
$\epsilon$ with $\epsilon'$. \label{item-C}
\end{itemize}
This condition implies that the corresponding approximating graph will
be connected.

What does an edge $(\gamma\epsilon,\gamma\epsilon')\in \HE_{\mathrm{tr},n}$ stand for?
Let's consider first the case $n=0$ (so there is no $\gamma$).
The two paths $\tau(\epsilon),\tau(\epsilon')\in\Pi_\infty(\Gg)$ 
both start at vertex $v$ and differ on their first edge. 
At some level $n_{(\epsilon,\epsilon')}$ they become equal again; let's say that
$v_{(\epsilon,\epsilon')}\in\Gv$ is the vertex at which this happens. 
The edge ${(\epsilon,\epsilon')}$ therefore defines a pair $(\eta,\eta')$ of paths of
length $n_{(\epsilon,\epsilon')}$ which have the same source and the same range vertex
$v_{(\epsilon,\epsilon')}$ and otherwise differ on each edge;
notably $\eta$ and $\eta'$ are the 
first $n_{(\epsilon,\epsilon')}$ edges of $\tau(\epsilon)$ and $\tau(\epsilon')$, respectively.
So the information encoded by ${(\epsilon,\epsilon')}$ is the $n_{(\epsilon,\epsilon')}$-supertile
corresponding to $v_{(\epsilon,\epsilon')}$ together with the position of two tiles of type $t_{v}$ 
encoded by the paths $\eta$ and $\eta'$. 
Let us denote by $r_{(\epsilon,\epsilon')}\in \R^d$ the vector of translation from the
first to the second tile of type $t_{v}$.   The remaining common part of the
paths  $\tau(\epsilon)$ and $\tau(\epsilon')$ (eventually an infinite repetition of $\epsilon^\ast$) places the  $n_{(\epsilon,\epsilon')}$-supertile
corresponding to $v_{(\epsilon,\epsilon')}$ into some translate of
$T^*$. 
And $r_{(\epsilon,\epsilon')}$ does not depend on this part.   

Now the situation for $n>0$ is similar, the only difference
being that the paths $\tau(\gamma\epsilon)$ and
$\tau(\gamma\epsilon')$ now split at the $n$-th vertex and meet for the
first time again at level $n+n_{(\epsilon,\epsilon')}$.  
If we denote by $r_{(\gamma\epsilon,\gamma\epsilon')}\in \RM^d$ the
translation vector between the encoded tiles then, due to self-similarity, one has: 
\begin{equation}
\label{eq-transtr}
r_{(\gamma\epsilon,\gamma\epsilon')}= \theta^{|\gamma|} r_{(\epsilon,\epsilon')}.
\end{equation}
See Figure~\ref{fig-chair3} for an illustration.
\begin{figure}[htp]
\begin{center}
\includegraphics[scale=0.4]{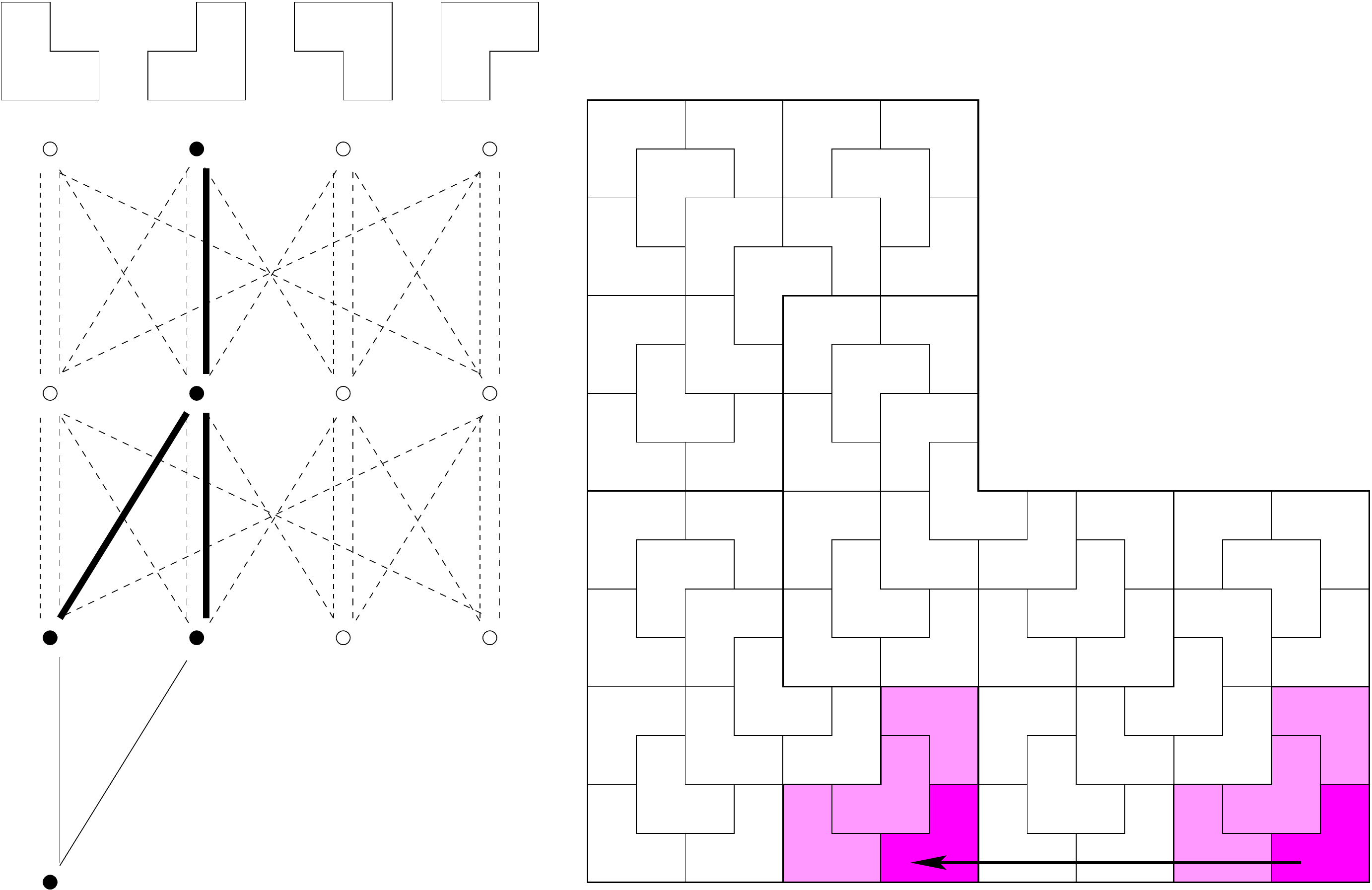}
\caption{{\small A doubly pointed pattern associated with a horizontal
    arrow $h\in \Hh_{\mathrm{tr},3}$. 
The arrow represents the vector $r_h$.
Here $n=2$ (the paths have lengths $2$), and $n_h=1$ (the paths join
further down at level $n+n_h=3$).}}  
\label{fig-chair3}
\end{center}
\end{figure}

We take $(\delta_n)_n$ of exponentially decreasing form:
$\delta_n =\rtr^n$, where $\rtr\in(0,1)$ is a parameter
which may be adapted. We moreover fix a choice function $\tau$.
Theorem~\ref{theorem-STapgraph} provides us with a self similar spectral triple $(C(\Xi_{t_v}),D_{\mathrm{tr}}^v,\HS_{\mathrm{tr}}^v)$ for
the algebra $C(\Xi_{t_v})$. Since $q_v^+$ is a homeomorphism we could also
use the version of the spectral triple of Theorem~\ref{theorem-STbdry}.
We call the triple the {\em transverse substitution spectral triple for the prototile
  $t_v$} of the substitution tiling.

\begin{definition}
By a {\em transverse substitution spectral triple} of a substitution tiling we mean
the direct sum over $v\in\Gv$ of the transverse spectral triples for the
prototiles $t_v$ defined as above. \end{definition}
The spectral triple depends on a choice of fundamental
horizontal edges $\hat\Hh_\mathrm{tr}$ satisfying condition (C), a parameter $\delta_{tr}$
determining the length function, and a choice function.
Since $\hat\Hh_\mathrm{tr}$ satisfies condition (C) above
and $(\delta_n)_n$ is exponentially decreasing~\cite{KS11}, the Connes distance induces the
topology of $\Xi_v$. 
A transverse substitution spectral triple is thus a second spectral triple for the canonical transversal $\Xi$.

\subsubsection{Longitudinal spectral triples for the prototiles of a substitution tiling}
\label{ssec-lgST}
We now use the negative part of $q_v$, ${q_v^-}:
\Pi^v_{-\infty,0}(\Gg)\to t_v$ to construct a spectral triple which we
call longitudinal. 
Notice that $\Pi^v_{-\infty,0}$ can be identified with
$\Pi^v_\infty(\tilde\Gg)$, where $\tilde\Gg$ is the graph obtained from
$\Gg$ by reversing the orientation of its edges: one simply reads
paths backwards, so follows the edges along their opposite
orientations. 
We choose a subset
$\hat\Hh_{\mathrm{lg}}\subset  \left\{ (\tilde
  \varepsilon,\tilde \varepsilon') \in \tilde\Ge\times\tilde\Ge \, :
  \, \tilde \varepsilon\neq \tilde \varepsilon', \; s(\tilde
  \varepsilon)=s(\tilde \varepsilon') \right\} $
  again satisfying condition (C).
To obtain the interpretation of 
a longitudinal horizontal edges it is more useful to work with
reversed orientations, that is,  
view 
$\hat\Hh_{\mathrm{lg}}\subset      \left\{
  (\varepsilon,\varepsilon') \in \Ge\times\Ge \, : \, \varepsilon\neq
  \varepsilon', \; r(\varepsilon)=r(\varepsilon') \right\}$, 
  as this was the way the Robinson map $\Rob$ was defined.  Then
  $(\varepsilon,\varepsilon')$ with $r(\varepsilon) = r(\varepsilon')$
  determines a pair of microtiles $(t,t')$ of type $s(\varepsilon)$
  and $s(\varepsilon')$, respectively, in a tile of type
  $r(\varepsilon)$. 
The remaining part of the double path
$(\tilde\tau(\varepsilon),\tilde\tau(\varepsilon'))$ serves to fix a
point in the two microtiles.  
Of importance is now the vector of translation $a_{(\epsilon,\epsilon')}$
between the two points of the microtiles.  

Similarily, an edge in $\HE_{\mathrm{lg},n}$ will describe a pair of $(n+1)$-th
order microtiles in an $n$-th order microtile.  
By self-similarity again, the corresponding translation vector $a_{(\gamma\epsilon,\gamma\epsilon')}\in
\RM^d$ between the two $(n+1)$-th order microtiles will satisfy 
\begin{equation}
\label{eq-translg}
a_
{(\gamma\epsilon,\gamma\epsilon')}= \theta^{-|\gamma|} a_{(\epsilon,\epsilon')}\,.
\end{equation}
See Figure~\ref{fig-chair4} for an illustration.
\begin{figure}[htp]
\begin{center}
\includegraphics[scale=0.5]{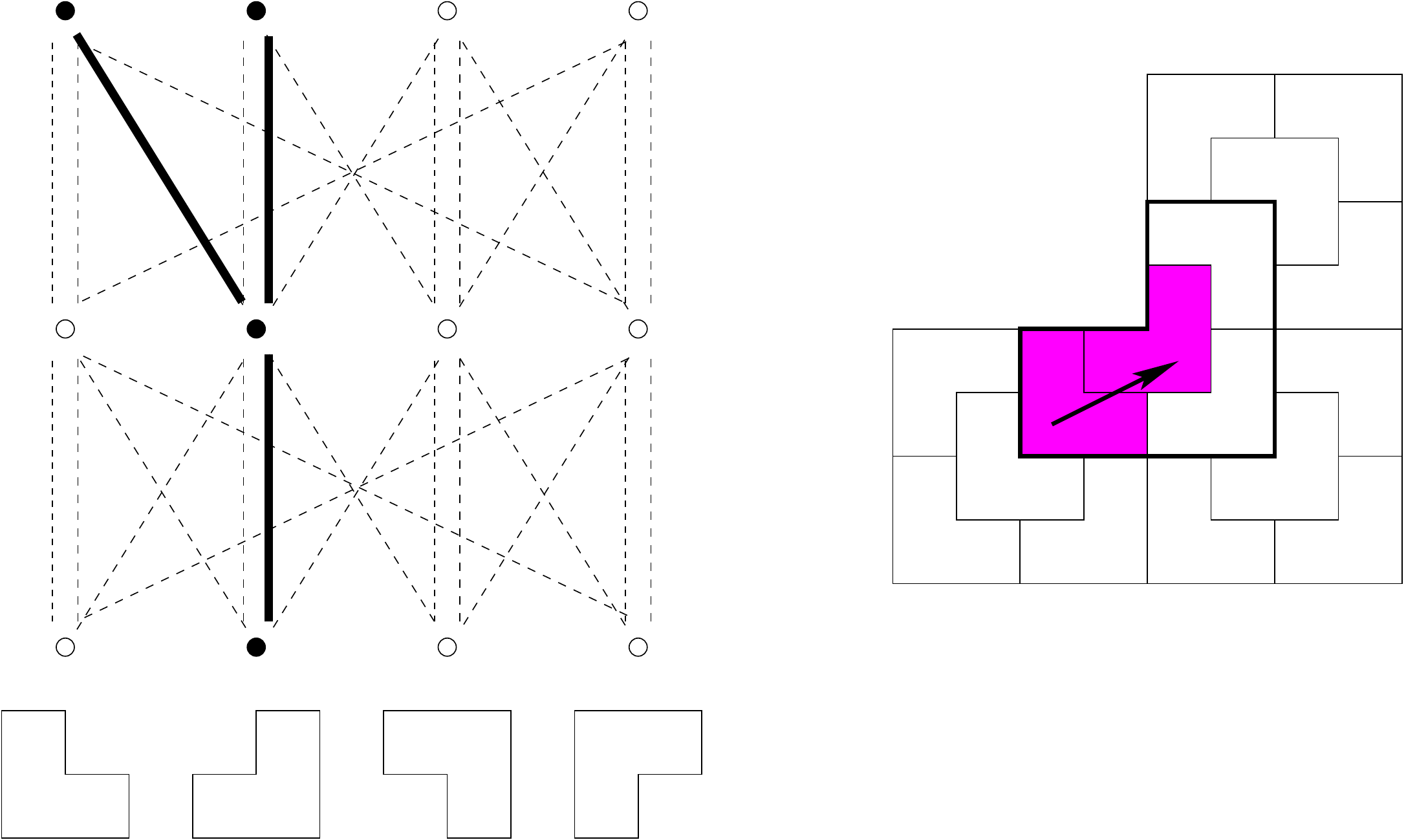}
\caption{{\small A microtile pattern associated with a horizontal
    arrow $h\in \Hh_{\mathrm{lg},2}$ (the pattern shown has the size of a
    single tile). The arrow represents the vector $a_h$.}} 
\label{fig-chair4}
\end{center}
\end{figure}

Again we
use an exponential decreasing length function $\delta_n=\rlg^n$, where $\rlg$ is
another parameter. Moreover we set ${\tilde \varepsilon}^* =
\varepsilon^*$ and choose a choice function $\tilde\tau$. 
$q_v^-$ is however not injective and the topology of 
$\Pi_{-\infty,0}^v$ differs from that of $t_v$ (the first is totally
disconnected whereas the second is connected). It is therefore a
priori not clear that Theorem~\ref{theorem-STapgraph} provides us
with a spectral triple for the algebra $C(t_v)$. However, (with the notation 
$C(\partial\Tt)_0$ introduced in eq.~\ref{equation-dense-subalgebra})
it can be shown that ${q^*}^{-1}(C( \Pi_{-\infty,0}^v)_0)$ contains all
functions over $t_v$ which are H\"older continuous with exponent
$\alpha = \frac{-\log(\rlg)}{\log(\theta)}$.
Thus Theorem~\ref{theorem-STapgraph} provides us with a spectral triple
$(C(t_v),D_{\mathrm{lg}}^v,\HS_{\mathrm{lg}}^v)$ for the prototile $t_v$.
\begin{definition}
By a {\em longitudinal substitution spectral triple} for the prototile $t_v$ we mean
a spectral triple as defined above. 
\end{definition}
The spectral triple depends on a choice of fundamental
horizontal edges $\hat\Hh_\mathrm{lg}$ satisfying condition (C), a parameter $\delta_{lg}$
determining the length function, and a choice function.
It should be noted that the Connes distance of this spectral triple
does not induce the topology of $t_v$.

\subsubsection{The full substitution spectral triple of a substitution tiling}
\label{ssec-SThull}
We now combine the above triples to get a spectral triple
\((C(\OP),\HS,D)\) for the whole tiling space
$\OP$.  
The graphs $\Gg$ and $\tilde\Gg$ have the same set of vertices $\Gv$,
so we notice that the identification  
\[
\Pi_{-\infty,+\infty}(\Gg) 
= \bigcup_{v\in\Gv} \Pi_{-\infty,0}^v(\Gg) \times \Pi_{0,+\infty}^v(\Gg)  
= \bigcup_{v\in\Gv}   \Pi_{\infty}^v(\tilde\Gg) \times \Pi_{\infty}^v(\Gg)
\]
suggests to construct the triple for $\OP$ as follows: first
we can use the tensor product construction for spectral triples to
obtain a spectral triple for $C(t_v\times \Xi_{t_v})\cong
C(t_v)\otimes C(\Xi_{t_v})$ from the two spectral triples considered
above. 
Furthermore, the $C^\ast$-algebra $C(\OP)$ is a subalgebra
of $\bigoplus_{v\in\Gv} C(t_v\times \Xi_{t_v})$ and so the direct sum
of the tensor product spectral triples for the different tiles $t_v$
provides us with a spectral triple for $C(\OP)$. Its Hilbert space,
representation and Dirac operator are given by
\begin{equation}
 \label{eq-STOmega}
\HS=\bigoplus_{v\in\Gv} \HS_\mathrm{tr}^v \otimes \HS_\mathrm{lg}^v\,, \quad
\pi=\bigoplus_{v\in\Gv} \pi_\mathrm{tr}^v\otimes \pi_\mathrm{lg}^v\,, 
\quad D= \bigoplus_{v\in\Gv} \bigl( D_\mathrm{tr}^v\otimes \id + \chi \otimes D_\mathrm{lg}^v  \bigr)\,,
\end{equation}
where $\chi$ is the grading of the transversal triple (which comes
from flipping the orientations of the edges in $\hat\HE_\mathrm{tr}$). 
\begin{definition}
By a {\em full substitution spectral triple} of a substitution tiling we mean a spectral triple for $\OP$ as defined above. 
\end{definition}
The spectral triple depends on a choice of transverse and longitudinal fundamental
horizontal edges $\hat\Hh_\mathrm{tr}, \hat\Hh_\mathrm{lg}$ satisfying condition (C), two parameters $\delta_{tr},\delta_{lg}$
determining the length function, and a choice function.

\section{Zeta functions and spectral measures}
\label{sec-zeta}

The Dirac $D$ operator of a spectral triple is supposed to have
compact resolvent.
As we suppose for simplicity that $0$ is not in its spectrum, the sequence of eigenvalues of $|D|^{-1}$, ordered
decreasingly and counted with their multiplicity, tends to zero at infinity.
In the construction of Section~\ref{sec-ag} the eigenvalues are given by the length function so that the
zeta function for the spectral triple of  Section~\ref{sec-ag} is
formally given by the series 
\begin{equation}
\label{eq-zeta-H}
\zeta(z) =  \sum_{(v,v')\in \HE}
  \delta(v,v')^z\,,
\end{equation}
or even by
\begin{equation}
\label{eq-zeta-Hn}
\zeta(z) = \sum_{n=1}^\infty \#\HE_n\; \delta_n^z
\end{equation}
in case when the edges can be written $\HE=\bigcup_n \HE_n$ such that the
length function depends only on the level $n$ of the edge, {\it i.e.\,} $\delta(v,v')=\delta_n$ for all $v,v'\in \HE_n$.
In the latter case $\zeta$ is manifestly independent of the choice
function (however notice that in Palmer's construction the length function
depends on the choice function and hence so does $\zeta$).

We suppose that the series has a finite abscissa of
convergence $s_0>0$, i.e.\ is convergent for $z$ with $\Re(z)>s_0$. 
This so-called metric dimension $s_0$ has a certain significance 
for the examples we discussed. There is a general expectation that a
good spectral triple for a metric space has a metric dimension which
coincides with a dimension of the space. Such a result holds for the
spectral 
triple of Section~\ref{sec-ultra} for a compact ultrametric space when the
minimal choice for the horizontal edges is employed. 
This is a case in which the length function depends only on the level and so the
zeta function is independent of the choice function. 
\begin{theorem}[\cite{PB09}] 
Let $(X,d)$ be a compact ultrametric space.
If its associated Michon tree has uniformly bounded branching\footnote{the number of edges in $e\in \Tt^{(1)}$ with source vertex $s(e)=v\in \Tt^{(0)}$ is uniformly bounded in $v$.},
then the Pearson--Bellissard spectral triple has metric dimension equal to the upper box dimension of $X$. 
\end{theorem}
Similarly, the spectral measure, defined on functions
$f\in C(X)$ by 
$$
\Tt(f) = \lim_{s\to s_0^+} \frac1{\zeta(s)} \mbox{\rm Tr}(|D|^{-s}\pi(f))
$$
is not just any measure but yields a measure well-known to the
cases. We have already mentionned Palmer's theorem (Theorem~\ref{thm-Palmer})
expressing the fact that for a compact metric space one can construct
a spectral triple for which $s_0$ is the Hausdorff dimension, and
the spectral measure is the Hausdorff measure.

\subsection{The metric dimension and complexities}
\label{ssect-complexity-ranks}
In the context of ordinary transverse spectral triples defined for tilings (based on the tree of patches), or of spectral triples defined for subshifts (based on the tree of words), the metric dimension is related to complexities.

Recall that we call $r$-patch a patch of the tiling which has a
puncture at the origin, and just covers $B_r(0)$, that is, all tiles of the
patch intersect $B_r(0)$ non-trivially. As we assume that our tiling is FLC we can define the 
function \(p : \RM^+ \rightarrow \NM\)
associating to $r$ the number of $r$-patches. It is called the {\em patch counting} (or {\em complexity}) function. This function might increase exponentially but we are here interested in tilings where it grows only polynomially, a feature which can be interpreted
as a sign of order in the tiling.  
We call
\begin{equation}
\label{eq-weakcomp}
\blo = \sup\{\gamma \, :\,  p(r) \ge r^\gamma \; \text{\rm for large } r \}\,, \qquad
\bup = \inf\{\gamma : \, p(r) \le r^\gamma \; \text{\rm for large } r\}\,.
\end{equation}
the lower an upper complexity exponents.
Notice that these exponents can be alternatively defined as the
$\liminf$ and $\limsup$ of $\log p(r) / \log r$ respectively. 
In case both are equal we call their common value the 
{\em weak complexity exponent}. 
If $p(r)$ is equivalent to $c r^\beta$ for some constant $c>0$,
then $\beta$ is simply called the complexity exponent. This implies
that $\blo = \bup = \beta$ but is stronger than existence of the weak exponent.

One could imagine
different definitions of $r$-patches using other geometric objects
than balls or other norms on $\R^d$ than the euclidean one. This would
effect the function $p(r)$ but not the exponents. 

The weak complexity exponents have an important interpretation.
If one takes $\delta(r)=\frac{1}{r}$ in the definition for the ultra
metric of the transversal $\Xi$ of an FLC tiling, then the weak complexity
exponents are the lower and upper box dimensions of $\Xi$ \cite{Jul09}: 
\[
\blo = \underline{\dim}(\Xi) \,, \qquad \bup = \overline{\dim}(\Xi)\,. 
\]
In the case of (primitive) substitution tilings of $\RM^d$, the
complexity exponent exists
\cite{JS10} and is equal to the
Hausdorff dimension of $\Xi$ \cite{JS11}, we comment on that again
further down. 



\subsubsection{The case of ordinary transverse spectral triples for tilings}
Let $a(v)$ be  the {\em branching
  number} of the vertex $v$ {\bf minus one}, so $a(v)+1$ is the
number of edges in $\Te$ which have source $v$.  
We define the $k$-th related Dirichlet series 
\[
\zeta_k(s) = \sum_{v\in\Tt^{(0)}} a(v)^k \; \delta(r_v)^s \,, k\in \NM\,,
\]
where $r(v)$ is the radius of the patch associated with the vertex $v\in\Tt^{(0)}$, and we use the convention $0^0=0$.
Then the zeta functions of the ordinary transverse spectral triples constructed from $\HE$ in Section~\ref{ssec-STtrans} are given as follows: 
\begin{itemize}
\item $\zeta^{\rm max}:=(\zeta_1+\zeta_2)/2$ is the zeta function if we take the
maximal choice $\HE=\HE^{\rm max}$;
\item $\zeta^{\rm min}:=\zeta_0$ is the zeta function for a
minimal choice $\HE=\HE^{\rm min}$. 
\end{itemize}
We  denote by $s_0^{\rm min/max}$ the abscissa of
convergence of $\zeta^{\rm min/max}$.

\begin{theorem}[\cite{KS11}]
Consider the ordinary transverse spectral triple of a $d$-dimensional tiling of finite local complexity (Def.~\ref{def-STtrans}). 
Assume that the function $\delta$ belongs to \(L^{1+\epsilon}([0,\infty))
\setminus L^{1-\epsilon}([0,\infty))\) for all $\epsilon$ small
enough. Then 
\[
\blo \le s_0^{min} \le 
s_0^{max} \le \bup + d-1 \,.
\]
\end{theorem}
The theorem applies for instance to spectral triples constructed using $\delta(r) = \frac1{r+1}$ 
as function. The latter does not belong to $L^{1}([0,\infty))$ but to $L^{1+\epsilon}([0,\infty))$ for all $\epsilon >0$.

\subsubsection{The case of spectral triples defined for subshifts}
\label{ssec-STsubshift-lapl}

The complexity function for an infinite word is $p(n)=$ number of distinct factors of length $n$.
We define further the {\em right-special complexity} $\prs$ as:
\begin{equation}
\label{eq-rscomplexity}
\prs(n) = \text{\rm number of distinct right-special factors of length $n$} \,, 
\end{equation}
as well as the {\em privileged complexity} $\ppr$ as:
\begin{equation}
\label{eq-prcomplexity}
\ppr(n) = \text{\rm number of distinct privileged factors of length $n$} \,.
\end{equation}
One defines weak complexity exponents $\bpr$ and $\brs$ for $\ppr$ and
$\prs$ just as in equation~\eqref{eq-weakcomp} for $p$. 
Then we have (Corollary~6.2 in \cite{KLS11}) 
\[
\bpr \le \brs = \beta -1 \,.
\]
When do we have equality between the exponents $\bpr = \brs$?
This question is related to the notion of {\em almost finite rank}. 
Given an $r$-patch $q$ of an FLC tiling $T$, the Delone set $\Ll_q$ of
occurrences of $q$ in $T$ can be tiled by Voronoi cells with
finitely many prototiles. Let $n(q)$ be this number. 
If $n(q)$ is bounded in $q$, then the tiling $T$ is
said to have {\em finite rank}. 
If there are constants $a,b>0$ such that $n(q) \le a
\log(r(q))^b$,
where $r(q)$ is the size of the patch, then $T$ is said to have {\em almost finite rank}. 

We denote by $\zeta^{\rm pr}$ the zeta function associated with the
privileged choice of horizontal edges.
\begin{theorem}[ \cite{KLS11}]
\label{prop-complex}
Consider the spectral triple of a one-sided subshift as in Def.~\ref{def-STsubshift}. 
If the subshift is repetitive, has almost finite rank, and \(\delta \in
L^{1+\epsilon}([0,\infty)) \setminus L^{1-\epsilon}([0,\infty))\) for
all $\epsilon$ small enough, then $\zeta^{\rm max}$ and
$\zeta^{\rm pr}$, as well as the related Dirichlet series $\zeta_k$, have the
same abscissa of convergence. If furthermore the subshift admits weak complexity
exponents, then $\bpr = \brs = \beta -1$, and $\beta$ coincides with
the common abscissa of convergence.
\end{theorem}
As an aside we mention that the equality $\bpr = \brs$
can be seen as an asymptotic version of a stronger
results  which holds for rich words \cite{GJWZ09}:
\(\ppr(n)+\ppr(n+1)=p(n+1) -p(n)+2\).
Indeed one has \(\ppr(n) \le p(n+1)-p(n) \le (|\Aa|-1) \ppr(n)\), and thus $\bpr=\brs$. And for rich words, privileged factors and palindromes are equivalent.

\paragraph{Example}
For the Fibonacci subshift of Example~\ref{ex-Fibo}, the tree of words $\Tt$ (shown at the end of Section~\ref{ssec-STsubshift}) has a single branching vertex in $v_n^\ast \in \Tt^{(0)}_n$ for each $n$, whose branching number is exactly $2$ (this is a property shared by all Sturmian words). 
That is $a(v^\ast_n)=1$ for all $n$, and $a(v)=0$ for all other $v\neq v^\ast_n$ for all $n$.
We choose the weight $\delta$ so that it only depends on the length of the factors: for instance $\delta(v)=\frac{1}{|v|}$ to satisfy the hypothesis of Theorem~\ref{prop-complex}.
Then for all $k\in \NM$ the zeta functions are all equal to 
\[
\zeta_k(z) = \sum_{n\ge 1} \sum_{v\in \Tt^{(0)}_n} a(v)^k \delta(v)^z = 
\sum_{n\ge 1} \delta(v_n^\ast)^z = \sum_{n\ge 1} \frac{1}{n^z},
\]
whose abscissa of convergence is $1$. This is also the complexity exponent of the Fibonacci subshift (all Sturmian words have indeed complexity function $p(n)=n+1$).

\subsection{The zeta function for a self-similar spectral triple} 

We determine in more detail the form of the zeta-function for the
triple associated with a subhift of finite type with self-similar
choices of horizontal edges, length function and choice function. The
most important of these choices is the scale $\rho\in(0,1)$, i.e.\ the
parameter such that $\delta_n = \rho^n$. Let $A$ be the
{\em graph matrix} of $\Gg$, that is, the matrix $A$ with coefficients
$A_{vw}$ equal to the number of edges which have source $v$ and range
$w$. 
The number of paths of length $n$ starting from $v$ and ending in $w$ is then ${A^n}_{vw}$.
We require that $A$ is {\em primitive}: \(\exists N \in \NM, \, \forall v,w, \, A^N_{vw}>0\).
Under this assumption, $A$ has a non-degenerate positive eigenvalue $\pf$ which is strictly larger than the modulus of any other eigenvalue. 
This is the Perron-Frobenius eigenvalue of $A$. 
Let $\lambda_1, \lambda_2, \ldots \lambda_p$ be the eigenvalues of $A$, ordered by decreasing modulus, and with $\lambda_1=\pf$.
We now compute the zeta function $\zeta(z)$ as in equation~\eqref{eq-zeta-Hn}.
The cardinality of $\#\Hh_n$ can be estimated as 
\[
\#\Hh_n=\sum_{\substack{v,w\in\Gv}} A^{n-1}_{vw} n_w = \sum_{j=1}^p C^j_{\hat\Hh} \lambda_j^n + o(|\lambda_p|) \,,
\]
where $n_w =\#\{(\epsilon,\epsilon')\in\hat\HE:s(\epsilon)=s(\epsilon')=w\}$, and the $C^j_{\hat\Hh}$,  $1\le j \le p$, are constants.
Hence each eigenvalue of $A$ gives a geometric contribution to the Dirichlet series $\zeta(z)$.

\paragraph{Example}
For the Fibonacci subshift of Example~\ref{ex-Fibo}, the graph matrix reads $A=\begin{pmatrix} 2 & 1 \\ 1 & 1 \end{pmatrix}$, with eigenvalues $\lambda_1=\tau^2$ (the Perron--Frobenius), and $\lambda_2=\tau^{-2}$, where $\tau=(\sqrt{5}+1)/2$ is the golden mean.
We choose $\hat{\Hh}= \hat{\Hh}_{\rm max}$, so that
\[
\# \Hh_n =  \begin{pmatrix} 1 & 1 \end{pmatrix}  \begin{pmatrix} 2 & 1 \\ 1 & 1 \end{pmatrix}^{n-1}
\begin{pmatrix} 3 \\ 2 \end{pmatrix} = (2+\tau) \, \tau^{2n} + 5 \tau^{-2n},
\]
(we counted only unoriented edges here) and we have exactly
\[
\zeta(z) =\sum_{n\ge 1} \# \Hh_n \ \rho^{nz}= \frac{2+\tau}{1-\tau^2\rho^z} + \frac{5}{1-\tau^{-2}\rho^z}.
\]
So the spectral dimension of the self-similar spectral triple for the Fibonacci tiling is $s_0=\frac{2 \log \tau}{-\log \rho}$.

In general we can show the following:

\begin{theorem}
\label{thm-zeta} Consider a self-similar spectral triple as in
Def.~\ref{def-ST-self-similar} with scale $\rho$.  
Suppose that the graph matrix is diagonalizable with eigenvalues
$\lambda_j$, $j=1, \ldots p$.  
The zeta-function $\zeta$ extends to a meromorphic function on $\CM$
which is invariant under the translation \(z \mapsto z+\frac{2\pi
  \imath}{\log \rho}\). 
It is given by
\[
 \zeta(z) = \sum_{j=1}^p \frac{C_{\hat\HE}^j }{1-\lambda_j\rho^z} + h(z)
\]
where $h$ is an entire function. 
In particular $\zeta$ has only simple poles which are located at
\(\{\frac{\log\lambda_j+2\pi \imath k}{-\log \rho}: k\in \ZM,
j=1,\ldots p\}\) with residus given by 
\begin{equation}
\label{eq-s0res}
\mbox{\rm Res}(\zeta,\frac{\log\lambda_j+ 2\pi \imath k}{-\log
  \rho})=\frac{ C_{\hat\HE}^j  \lambda_j}{-\log\rho} \,. 
\end{equation}
In particular, the metric dimension is equal to $s_0 = \frac{\log\pf}{-\log \rho}$.
\end{theorem}
\begin{rem}
{\em The periodicity of the zeta function with purely imaginary period
whose length is only determined by the factor $\rho$ comes from the
self-similarity of the construction. It is a feature 
which distinguishes our spectral triples from known
triples for manifolds. Note also that $\zeta$ may have a (simple) pole
at $0$, namely if $1$ is an eigenvalue of the graph matrix $A$.} 
\end{rem}
\begin{rem}
{\em The location of the poles and hence also the metric dimension does not
depend on the choice of $\hat\HE$ (neither on the choice function).}
\end{rem}
\begin{rem}
\label{rem-zeta}
{\em In the general case, when $A$ is not diagonalizable, it is no longer
true that the zeta-function has only simple poles. Indeed, at least if $|\lambda_i|\neq 1$
the zeta function $\zeta(z)$ has poles of order $m_j$ at \(z=\frac{\log\lambda_j+ 2\pi
  \imath k}{-\log \rho}\)
where $m_j$ is the size of the largest Jordan block of $A$ corresponding to
eigenvalue $\lambda_j$.}
\end{rem}

Theorem \ref{thm-zeta} applies to the transversal and the longitudinal spectral triples discussed in Section~\ref{ssec-trans} and \ref{ssec-lgST}.
In particular it applies to the spectral triple used in \cite{JS10}.
It is the transverse substitution spectral triple of Section~\ref{ssec-trans} with minimal choice for the horizontal edges
and $\rho_{tr} = \theta^{-1}$.
This shows that the transverse metric dimension for a substitution tiling of $\RM^d$ equals $d$.

The metric dimension is additive under the tensor product construction
of spectral triples. Hence the metric dimension of the spectral
triples for $\OP$ from Section~\ref{ssec-SThull} is 
$$s_0 = d\Bigl( \frac{\log\theta}{-\log \rho_{tr}} +\frac{\log\theta}{-\log \rho_{lg}} \Bigr). $$
\bigskip


\section{Laplacians}
\label{sec-laplace}

\newcommand{\tr}{\text{\rm tr}}
\newcommand{\str}{\text{\rm str}}

Recall that a spectral triple $(C(X),D,\HG)$ for a
compact space $X$ together with a state $\Tt$ on the Hilbert space $\HG$
defines a quadratic form 
$(f,g) = \Tt \bigl([D,\pi(f)]^\ast [D,\pi(g)]\bigr)$ which 
may be interpreted as the analog of a Laplacian on
$X$. This is of particular interest if $X$ does not carry an a priori
differentiable structure. 
 
But the Laplacian does not come for free. In fact, subtle analytic
questions have to be resolved which depend on further choices. 
The question is whether the quadratic form can be extended to 
 a quadratic form on the Hilbert space $L^2_\RM(X,\mu)$, where $\mu$ 
is typically (but not necessarily) the spectral measure. This
problem involves the determination of a core for the form (a dense
subspace on which it can be defined, and then extended by a closure
operation). The resulting form will a priori depend on the
choice for the core. One then has to check that the closure of the form has
the desired properties of a Dirichlet form.

Although the usual procedure of construction of such forms employs the
spectral state $\Tt$ we will proceed at first in a
slightly different manner. Let $\tr$ be  a trace
on the Hilbert space $\HG$ and consider the bi-linear
form 
\begin{equation}
\label{eq-form0}
(f,g) \mapsto \tr \bigl(|D|^{-s} [D,\pi(f)]^\ast [D,\pi(g)]\bigr) \,
\end{equation}
for $f,g\in C(X)$ such that $|D|^{-s} [D,\pi(f)]^\ast [D,\pi(g)]$ is trace class.
We consider two cases here. 
In both cases the spectral triple is defined by an approximating graph $G_\tau=(V,E)$
as in Section~\ref{sec-ag}.
\begin{itemize}
\item[(1)] We consider in  Section~\ref{ssec-PBLap}
{\em $\tr=\TR$ the usual operator trace on
    $\HG$} and leave $s$ as a parameter. 
Upon averaging over choice functions $\tau$ we obtain the
Pearson--Bellissard Dirichlet form and its Laplacian. We discuss its
spectral theory in the case that $X$ is an arbitrary compact ultrametric space 
getting more concrete results for $X=\Xi$, the canonical 
transversal of a tiling space or a subshift. 

\item[(2)] In Section~\ref{ssec-TilLap} we employ 
{\em $\tr = \str$ the singular trace on $\HG$} at $s=s_0$.
This corresponds to using the spectral state $\Tt$ for the definition
of the quadratic form. We get very concrete results in the case that
$X=\OP$ is a continuous Pisot substitution tiling space, 
obtain two forms (and associated Laplacians): one of longitudinal
nature, and one of transversal nature. The two forms may be combined
into a single one.
\end{itemize}

\subsection{The operator trace and the  Pearson--Bellissard Laplacian}
\label{ssec-PBLap}

In this section we let $(X,d)$ be a compact ultrametric space.
We refer the reader to Section~\ref{sec-ultra} where the canonical Michon tree $\Tt$ is constructed.

We wish to extend the quadratic form \eqref{eq-form0} using the operator trace $\tr=\TR$.
The naive idea, namely to consider the graph Laplacian on the approximating
graph $G_\tau=(V,E)$ cannot 
work, as it is defined on $\ell^2(V)$ and no continuous non trivial
continuous function on $X$ restricts to $\ell^2(V)$. 
The way around this problem is to average over choice functions
$\tau$. For this we need a probability measure on the set of choice
functions. Let $\tau: \Tt^{(0)} \rightarrow \partial \Tt$ be a choice
function. We can either say that $\tau$ chooses an infinite extension
for every finite path, or we can say that it chooses the follow up
vertex for every vertex. As such, $\tau$ can be thought of  as an infinite
family of finite choices, namely for each vertex which one to choose
next. 
The family of choice functions can therefore be identified with the
product \(Y=\prod_{v\in\Tt^{(0)}} \Tt^{(0)}(v)\), where, we recall, $\Tt^{(0)}(v)$ is the finite set
of vertices following $v$ (hence can be ignored if $v$ is not branching). 
Furthermore, there is a one-to-one correspondence between Borel probability
measures on $X$ and product probability measures on $Y$ (\cite{KS11} Lemma 5.8).
We endow $Y$ with a product probability measure $\PM$, 
and let $\nu$ be the corresponding measure on $X$.
We divide equation~\eqref{eq-form0} by $2$ (to avoid counting
unoriented edges twice) and average over choices to define the
quadratic form on $L^2_\RM(X,\nu)$: 
\begin{equation}
\label{eq-traceform}
Q_s(f,g) = \frac{1}{2} \int_Y \TR \bigl( |D|^{-s}[D,\pi_\tau(f)]^\ast
[D,\pi_\tau(g)]\bigr) \, d\PM(\tau) \,. 
\end{equation}
It follows from Pearson--Bellissard's work that equation~\eqref{eq-traceform} defines a Dirichlet form whose domain is generated by (real valued) locally constant functions on $X$ of the form $\chi_v=\mathbb{1}_{q([v])}$, where $[v] \subset \partial \Tt$ is the set of all infinite paths going through $v$.
And one identifies $Q_s(f,g)$ with $\langle f, \Delta_s
g\rangle_{L_\RM^2(X,\nu)}$ for  the {\em Laplacian} $\Delta_s$: a
non-positive definite self-adjoint operator on $L_\RM^2(X,\nu)$, with
pure point spectrum. 

Let us comment on the form of equations~\eqref{eq-form0}
and \eqref{eq-traceform}.
Choice functions can be interpreted as analogue to tangent vectors
over the ``noncommutative manifold'' $X$, and hence $Y$ stands for the unit tangent
sphere bundle to $X$. 
The noncommutative gradient  of a function $\nabla_\tau f =
[D,\pi_\tau(f)]$ therefore stands for the directional derivative of
$f$ along $\tau$. 
The quadratic form \eqref{eq-traceform} is therefore reminiscent to 
the classical integral $\int_M \g(\nabla f,\nabla g) \,d \vol =
\langle f, -\Delta g\rangle_{L^2(M)}$ defining the Laplace--Beltrami
operator over a Riemannian manifold $(M,\g)$. The analogy is not perfect,
however, as integration of a function over the manifold, $\int_M f \,d
\vol$, usually corresponds to the application of the spectral state.

We now particularise the above construction to determine explicitly
$\Delta_s$.
We consider the {\em minimal choice} of horizontal edges, $\Hh=\Hh^{\rm min}$, as described in Section~\ref{sec-ultra}. 
We assume further that $\mu$ is the spectral measure, and that the length function satisfies: $\delta(h)=d(q\circ\tau(s(h)),q\circ\tau(r(h)))$, {\it i.e.\,} the distance in $X$ between the source and range of the image of $h\in \Hh$ in $X$. 
Then this yields the spectral triple of Pearson--Bellissard as in Definition~\ref{def-stultra}.
If all branching vertices of the tree $\Tt$ have exactly two outgoing vertical edges, then our general construction reduces to that of Pearson--Bellissard in \cite{PB09}.
For instance, this is the case for $X$ a Sturmian subshift as in Example~\ref{ex-Fibo},
with $\Tt$ its tree of words. 

We now average the form in equation~\eqref{eq-traceform} over minimal choices of edges $\Hh$, to obtain the Pearson--Bellissard Laplacian.
It can be determined explicitly and diagonalised \cite{JS10}.
There are two regimes: if $s< s_0+2$ then $\Delta_s$ is unbounded, 
while for $s>s_0+2$ it is bounded. (In the bounded case it is possible
to obtain an embedding of the transversal of substitution tiling space in a Euclidean space \cite{JS11}.)
The eigenvalues and eigenvectors are parametrized by the set $\Tt^{(0)}_{br} \subset \Tt^{(0)}$ of branching vertices:
a basis of eigenvectors of $\Delta_s$ is given by functions 
\[
\phi_v = \frac{1}{\mu[u]} \chi_{u} -\frac{1}{\mu[u']}\chi_{u'},
\quad v\in\Tt^{(0)}_{br}, \quad u\neq u'\in\Tt^{(0)}(v),
\]
where $\chi_v=\mathbb{1}_{q([v])}$ and $\mu[v]$ stands for the $\mu$-measure of $q([v])$.
The associated eigenvalues $\lambda_v$ can be calculated explicitly.
For $s=s_0$, the spectral dimension of $X$, one can compute the Weyl
asymptotics of the eigenvalues of $\Delta_{s_0}$: the number of
eigenvalues of modulus less than $\lambda$ behaves asymptotically like $ \lambda^{s_0/2}$, in
analogy with the classical case. 
\paragraph{Example} For the Fibonacci subshift of Example~\ref{ex-Fibo}, the eigenfunction, associated with the vertex $aba$ in its tree $\Tt$ shown at the end of Section~\ref{ssec-STsubshift}, reads
\[
\phi_{aba} = \frac{1}{\mu[abab]} \chi_{abab} - \frac{1}{\mu[abaa]} \chi_{abaa} .
\]

\medskip

We now turn to the self-similar case of substitution tilings: $X=\Xi_\Phi$ is the discrete transversal to a substitution tiling space with substitution map $\Phi$.
Let $A$ be the substitution matrix of the substitution $\Phi$, $\Gg$ its substitution graph, and $\Tt$ the associated tree (of super tiles), as in Section~\ref{ssec-STsubst} and~\ref{ssec-finitetype}.
The spectral measure $\mu$ of the corresponding transverse spectral triple (see Section~\ref{ssec-trans}) is the transverse invariant ergodic probability measure on $\Xi_\Phi$.

The Laplacian $\Delta_s$ is in this case completely and explicitly determined \cite{JS10}: all eigenvectors and corresponding eigenvalues $(\phi_v, \lambda_v), v\in \Tt^{(0)}_{br}$, can be computed algorithmically from the ones $(\phi_0, \lambda_0)$ associated with the root of $\Tt$, 
by use of the Cuntz-Krieger algebra $\Oo_A$ of the substitution matrix.
The  Cuntz--Krieger algebra $\Oo_A$ \cite{CK80} is the universal C$^\ast$-algebra generated by 
partial isometries $U_1, U_2, \ldots U_n$ of an infinite dimensional separable Hilbert space as follows: the operators $U_1, U_2, \ldots U_n$ are in one-to-one correspondence with the vertices $u_1, \ldots, u_n$ in $\Tt^{(0)}_1$ ({\it i.e.} the prototiles),  and subject to the relations 
\(U_i^\ast U_i = \sum_{j} A_{ij} U_j U_j^\ast\).
There are two faithful $\ast$-representations of $\Oo_A$: $\rho_1: \Oo_A \rightarrow L^2(X,\mu)$ and $\rho_2 : \Oo_A \rightarrow \ell^2(\Tt^{(0)})$ such that 
\[
\phi_v = \rho_1(U_{i_m} \ldots U_{i_2} U_{i_1}) \phi_0\,, \qquad 
\lambda_v = \langle \rho_2(U_{i_m}  \ldots U_{i_2} U_{i_1}) \delta_0, \delta_v \rangle \, \lambda_0\,,
\]
where 
we wrote $v=u_{i_m}, \ldots u_{i_2}, u_{i_1}$ the unique sequence of vertices in $\Tt^{(0)}$ from $v$ up to the root, the root vertex not being included.
Moreover, $\delta_0, \delta_v$ are the basis elements of $\ell^2(\Tt^{(0)})$ associated with the root and vertex $v$ and $\langle\cdot,\cdot \rangle$ denotes the scalar product in $\ell^2(\Tt^{(0)})$.

\subsection{The singular trace and Laplacians for substitution tilings}
\label{ssec-TilLap}
We now consider the quadratic form in \eqref{eq-form0} which we
obtain if we use for $\tr$
a singular trace $\str$ on $\HG$ at parameter $s=s_0$. A priori there
may be many singular traces, but if the limit below exists then,
apart from an overall factor, they all amount to:
\begin{equation}
\label{eq-straceform}
Q(f,g) = \str \bigl(|D|^{-s_0} [D,\pi(f)]^\ast [D,\pi(g)]\bigr) =  
\lim_{s\to s_0^+} \frac{1}{\zeta(s)} \TR \bigl( |D|^{-s} [D,\pi(f)]^\ast [D,\pi(g)]\bigr)\,.
\end{equation}
In other words \( Q(f,g) = \Tt \bigl([D,\pi(f)]^\ast [D,\pi(g)]\bigr)\) 
is defined by  the spectral state $\Tt$. 

Our main application is to a self similar spectral triple from a
substitution tiling, as discussed in Section~\ref{ssec-SThull}.
Note that in \eqref{eq-straceform} we do not integrate over choice functions. 
Such an integration would in fact not change much, as there are only
finitely many possible choices due to our self-similarity constraint on
the choice function.

Recall that \(C(\OP)\) is  a subalgebra of
\(\oplus_{t\in\Aa} C(t\times \Xi_{t})\) where $\Aa$ is the set of prototiles. 
We can employ the tensor product structure of $C(t\times \Xi_{t})\cong C(t)\otimes C(\Xi_{t})$  to decompose the form into the sum of a transversal form and
a longitudinal form : 
\begin{equation}
\label{eq-decspecstate}
Q(f,g) = 
Q_{lg}(f,g) + Q_{tr}(f,g)\,,
\end{equation}
where 
\begin{equation}
\label{eq-decspecstate-2}
Q_\alpha(f,g) = \Tt_\alpha\bigl( [D_\alpha,\pi_\alpha(f)]^\ast [D_\alpha,\pi_\alpha(g)] \bigr)\,.
\end{equation}
Here $\alpha = tr$ or $lg$ and the corresponding objects $\Tt_\alpha,D_\alpha,\pi_\alpha$ are the spectral state, the Dirac operator and the representation associated with the transverse and longitudinal substitution spectral triple of $t$ of Section~\ref{ssec-trans}
and~\ref{ssec-lgST}, respectively.
There are subtle technicalities behind this decomposition, in particular for the spectral state
(which holds only for so-called strongly regular operators on $\HG$
\cite{KS13}) which we will not discuss here. 

In both cases the operator 
$[D_\alpha, \pi_\alpha(f)]^\ast [D_\alpha, \pi_\alpha(g)]$ is diagonal and gives a contribution \( (\delta_e^\alpha f)^\ast \delta_e g\) on the edge $e\in E_n^\alpha=q\circ\tau \times q\circ \tau (\Hh_{\alpha,n})$, where
\begin{equation}
\label{eq-deltaef}
\delta_e^\alpha f = \frac{f(r(e)) -f(s(e))}{\rho_\alpha^n}\,. 
\end{equation}
Moreover
\[
Q_\alpha(f,g) = 
\lim_{s\rightarrow s_0^+} \frac{1}{\zeta_\alpha(s)}  \sum_{n\geq 1} \#E_n^\alpha \rho_\alpha^{ns} \; q_n^\alpha(f,g)
\quad \text{\rm with } \quad 
q^\alpha_n(f,g) = \frac{1}{\#E_n^\alpha }\sum_{e\in E_n^\alpha } \; \overline{\delta_e^\alpha f} \;
\delta_e^\alpha g\,. 
\]
Notice that $\lim_{s\rightarrow s_\alpha^+} \frac{1}{\zeta_\alpha(s)} \sum_{n\geq 1} \#E_n^\alpha \rho_\alpha^{ns} = 1$, and hence we have $Q_\alpha(f,g) = \lim_{n} q_n^\alpha(f,g)$ provided the limit exists.
We now briefly explain how we can evaluate these limits.
The following two paragraphs are a bit technical, the reader will find the main result stated and discussed in the last paragraph.

\paragraph{The longitudinal form}
Given a fundamental  edge $(\epsilon,\epsilon')\in\hat\Hh_{lg}$, recall 
that $a_{(\epsilon,\epsilon')}\in \RM^d$ denotes the translation vector between the punctures of the microtiles associated with $\epsilon$ and $\epsilon'$.
If $a_e \in \RM^d$ denotes the corresponding vector for $e\in E_{lg,n}$ of type $(\epsilon,\epsilon')$, 
that is, $e = q\circ\tau \times q\circ \tau(\gamma\epsilon,\gamma\epsilon')$ for some $\gamma$ of length $n$,
then by  \eqref{eq-translg} we have $a_e=\theta^{-n}a_h$, so $s(e) = r(e) + \theta^{-n} a_{h}$. 
We make for large $n$ the Taylor approximation in equation~\eqref{eq-deltaef}
\[
 \delta^{lg}_{e} f \simeq \left( \frac{\theta^{-1}}{\rho_{lg}} \right)^{n} 
	\ (a_h \cdot \nabla) f(s(e))\,.
\]
If $\rho_{lg}=\theta^{-1}$, and we further approximate the sum of  $(\delta^{lg}_{e} f)^\ast \delta^{lg}_{e} g$ over $e\in E_n^{lg}$ by a Riemann integral, then for $n$ large $q_n^{lg}(f,g)$ gives contributions of the form \(\int_t (a_h \cdot \nabla) \bar{f} \; (a_h \cdot \nabla) g \; d\mu_{lg}^t\) .

We define the operator on $L^2_{\RM}(\OP, d\mu)$: 
\begin{equation}
\label{eq-lgLaplace}
\Delta_{lg} = c_{lg} \nabla_{lg}^\dagger \Kk \nabla_{lg} \,, \quad \text{\rm with} \quad 
\Kk = \sum_{t\in \Aa, \,h\in\Hh^{lg}_1(t)} \freq(t) \ a_h \otimes a_h\,,
\end{equation}
where $\hat\Hh_{lg}(t)=\{(\epsilon,\epsilon'):s(\epsilon)=s(\epsilon')=t\}$,  $\freq(t)$ is the frequency of tile $t$, and we write $\nabla_{lg}$ for the {\em longitudinal gradient} on
$\OP$. The latter takes derivatives along the leaves of the foliation:
$\nabla_{lg}=\id\otimes \nabla_{\RM^d}$. This leads to the expression 
\begin{equation}
Q_{lg} (f,f)
= \left\{ \begin{array}{ll}
 \langle f, \; \Delta_{lg} \; f\rangle_{ L^2_{\RM}(\OP,
   d\mu)} & \mbox{if } \rho_{lg} =\theta^{-1} \\  
0 & \mbox{if } \rho_{lg} > \theta^{-1} \\
+\infty & \mbox{if } \rho_{lg} < \theta^{-1}
\end{array}\right. \,,
\quad \text{\rm for all } f\in C_{lg}^2(\OP)\,,
\end{equation}
where $C_{lg}^2(\OP)$ is the space of longitudinally $C^2$
functions on $\OP$. 
So we see that for $\rho_{lg}\ge \theta^{-1}$, $\Delta_{lg}$ is
essentially self-adjoint on the domain $C_{lg}^2(\OP)$, and
therefore the form $Q_{lg}$ is closable. 
For $\rho_{lg}<\theta^{-1}$ the form is not closable.

\paragraph{The transversal form}
Given a fundamental  edge $h=(\epsilon,\epsilon')\in\hat\Hh_{tr}$
let $t=s(\epsilon) = s(\epsilon')$ and denote as before by
$r_{h} \in \RM^d$ the return vector between the occurrences of $t$ in the
supertiles associated with $\epsilon$ and $\epsilon'$. 
If $r_e\in\RM^d$ denotes the corresponding vector for $e\in
E_{n}^{tr}$ of type ${(\epsilon,\epsilon')}$, then by self-similarity we have
$r_e=\theta^{n}r_h$, so $s(e) = r(e) + \theta^n r_{h}$. 

We assume now that the substitution is Pisot, {\it i.e.} $\theta$ is a Pisot number. Then we know that there are plenty of dynamical eigenfunctions: continuous functions $f_\beta$ satisfying \(f_\beta(\omega+r)=e^{2\imath\pi \beta(r)} f_\beta(\omega)\) for some $\beta\in{\R^d}^*$ and 
all $\omega\in \OP$ and $r\in \RM^d$. In fact, the set of $\beta \in{\R^d}^*$ for which such a function exists is a dense subgroup of $ {\R^d}^*$ and the Pisot substitution conjecture\footnote{discussed at length in  the contribution in the Pisot chapter}
states that if the substitution matrix is irreducible then the dynamical spectrum is purely discrete
which means that the eigenfunctions generate all of $L^2(\OP, \mu)$.
We assume this to be the case, and we choose the linear span of dynamical eigenfunctions to be the core of $Q_{tr}$.
We have
\[
 \delta^{tr}_{e} f_\beta =  \frac{1}{\rho_{tr}^n} \bigl( f_\beta(s(e) + \theta^n r_h) -f_\beta(s(e)) \bigr)=
 \frac{1}{\rho_{tr}^n} \left( e^{2\imath \pi \theta^n \beta(r_h) } -1 \right)  f(s(e))\,.
\]
By the arithmetic properties of Pisot numbers $\theta^n \beta(r_h)$ tends to an integer as $n$ goes to infinity. Moreover, the speed of convergence is governed by the Galois conjugates of $\theta$ of greatest modulus: $\theta_j, j=2, \cdots L$, $|\theta_j| = |\theta_2|$.
It follows that, for $n$ large, we have
\[
 \delta^{tr}_{e} f_\beta \simeq  \left(\frac{\theta_2}{\rho_{tr}}\right)^n 
 \left( \sum_{j=2}^L p_{\beta(r_h)} (\theta_j) \right)\   f(s(e))\,,
\]
where $ p_{\beta(r_h)}$ is some polynomial with rational coefficients.
We are left with summing the terms $(\delta^{tr}_{e} f)^\ast \delta^{tr}_{e} g$ over $e\in E_n^{tr}$, and approximate by a Riemann sum to get an asymptotically equivalent expression for $q_n^{tr}(f,g)$. 
There are averaging subtleties coming from the phases $\alpha_j$ of the $\theta_j$ and we have to assume that\footnote{The ratio $\log(\theta)/\log(|\theta_2|)$ is irrational unless $\theta$ is a unimodular
Pisot number of degree $J=3$ \cite{Wald}.}: \( \alpha_j -\alpha_{j'} + 2k \pi + 2\pi k' \frac{\log \rho_{tr}}{\log \rho_{lg}} \neq 0\), $\forall k,k'\in\ZM$.
Define the operator $\Delta_{tr}$ on the linear space of dynamical eigenfunctions by
\begin{equation}
\label{eq-trLaplace}
\Delta_{tr} f_\beta = - c_{tr} (2\pi)^2 \sum_{t\in\Aa, h\in \hat\Hh_{tr}(t) }\freq(t) \sum_{j=2}^L |p_{\beta(h)}(\theta_j)|^2 \ f_\beta \,.
\end{equation}
Then, on the space   of dynamical eigenfunctions the transversal form is given by 
\[
Q_{tr}(f_\beta,f_\beta)
 = \left\{ \begin{array}{ll}
\langle f_\beta,\Delta_{tr} f_\beta\rangle_{L^2(\OP,d\mu)}
 & \mbox{if } \rho_{tr} = |\theta_2| \\ 
0 & \mbox{if } \rho_{tr} > |\theta_2| \\
+\infty & \mbox{if } \rho_{tr} < |\theta_2|
\end{array}\right. \,.
\]
Clearly, $Q_{tr}$ is closable but trivial if $\rho_{tr}>|\theta_2|$, whereas   $Q_{tr}$ is not closable if
$\rho_{tr}<|\theta_2|$.

\paragraph{Main result and geometric interpretation}

We summarize here the results about the Dirichlet forms.
For a Pisot number $\theta$ of degree $J>1$, we denote $\theta_j, j=2,\cdots J$, the other Galois conjugates in decreasing order of modulus.
We write the subleading eigenvalues in the form
$\theta_j=|\theta_2|e^{\imath \alpha_j}$, $2\leq j\leq L$,  where $\alpha_j\in [0,2\pi)$. 
In particular, $|\theta_j|<|\theta_2|$ for $j>L$.
\begin{theorem}[\cite{KS13}] 
\label{thm-DirForm}
Consider a Pisot substitution tiling of $\RM^d$ with Pisot number $\theta$ of degree $J>1$.
Assume that for all $j\neq j'\leq L$ one has
\begin{equation}
 \label{eq-phasePisot}
 \alpha_j - \alpha_{j'} + 2\pi k + 2\pi \frac{\log|\theta_2|}{\log\theta} k' \neq 0\,, \quad \forall k,k' \in \ZM \,.
\end{equation}
Set $\rho_{lg} = \theta^{-1}$ and $\rho_{tr}=|\theta_2|$.

If the dynamical spectrum is purely discrete  then the set of finite linear combinations of dynamical eigenfunctions is a core for $Q$ on which it is closable.
Furthermore, $Q = Q_{tr} + Q_{lg}$, and $Q_{tr/lg}$ has generator $\Delta_{tr/lg}=\sum_{h\in \Hh_{tr/lg,1}}\Delta_{tr/lg}^h$ given by
\begin{eqnarray*}
\Delta_{lg}^h f_\beta &=&  -c_{lg}(2\pi)^2 \mbox{\em freq}(t_{h})  \beta(a_h)^2 f_\beta ,\\
\Delta_{tr}^h f_\beta &=&  -c_{tr}(2\pi)^2 \mbox{\em freq}(t_{h})  \langle \tilde{r_h}^\star,\beta\rangle^2 f_\beta 
\end{eqnarray*}
where $t_h$ is the tile associated with (the source of the vertical edges linked by) $h$, and the constants $c_{lg}$ and $c_{tr}$ depend only on the substitution matrix. 
\end{theorem}
We now explain the term $\tilde{r_h}^\star$  in the above equation, and give an interpretation of the Laplacians $\Delta_{tr/lg}$ as elliptic second order differential operators with constant coefficients on the maximal equicontinuous factor of the dynamical system $(\OP,\RM^d)$.

We assume for simplicity that $\theta$ is unimodular, see \cite{KS13} for the general case.
Our assumption that the dynamical spectrum is purely discrete is known to be equivalent to the tiling being  a cut-and-project tiling.
The maximal equicontinuous factor of the dynamical system $(\OP,\R^d)$ coincides with the $dJ$-torus $\TM$ of the torus parametrisation of the cut-and-project scheme.
The substitution induces a hyperbolic homeomorphism on that torus which allows us to split the tangent space at each point into a stable and an unstable subspace, $S$ and $U$.
The unstable tangent space is $d$-dimensional and can be identified with the space in which the tiling lives. We write $\tilde r$ for the vector in $U$ corresponding to $r$ via the identification of $U$ with the space in which the tiling lives.
The stable space $S$ can be split further into eigenspaces of the hyperbolic map, namely $S = S_2+S'$ where $S_2$ is the sum of eigenspaces of the Galois conjugates of $\theta$ which are next to leading in modulus ($\theta_j$ for $j=2, \ldots L$).
Finally ${}^\star:U\to S_2\subset S$ is the reduced star map.
This is Moody's star map followed by a projection onto $S_2$ along $S'$.

Since the dynamical spectrum is pure point and all eigenfunctions continuous the factor map $\pi:\OP\to\TM $ induces an isomorphism between $L^2(\Omega,\mu)$ and $L^2(\TM,\eta)$, where $\eta$ is the normalized Haar measure on $\TM$.
The Dirichlet form  $Q$ can therefore also be regarded as a form on $L^2(\TM,\eta)$. 
Now the directional derivative at $x$ along $u\in U\oplus S$ is given by  $$ (\langle u,\nabla\rangle f_\beta)(x)= \frac{d}{dt} f_\beta(x+t u)\left|_{t=0} \right. = 2\pi i \langle u,\beta\rangle f_\beta(x).$$ 
And we thus have
\begin{eqnarray*}
\Delta_{lg}^h &=&  c_{lg}\freq(t_{h})  \langle\tilde a_h,\nabla\rangle^2  ,\\
\Delta_{tr}^h  &=&  c_{tr}\freq(t_{h})  \langle\tilde{r_h}^*,\nabla\rangle^2 . 
\end{eqnarray*}
To summarize, the Dirichlet form  $Q$ viewed on $L^2(\TM,\eta)$ has as generator the Laplacian
$\Delta = \sum_{h\in\hat\HE_{lg}} \Delta_{lg}^h + \sum_{h\in\hat\HE_{tr}} \Delta_{tr}^h$ on $\TM$
which is a second order differential operator with constant coefficients containing (second) derivatives only in the directions $U + S_2$.

\section{Characterization of order}
\label{sect-charorder}

In this section we explain how noncommutative geometry can be used to
characterise combinatorial properties of tilings and subshifts.
These properties, equidistribution of frequencies and bounded powers,
are signs of aperiodic order. In this theory, which has been proposed in
\cite{KS11} and extended in \cite{KLS11},  
the interest is no longer Rieffel's question after 
the construction of a spectral triple whose associated Connes distance 
induces the topology, but a comparison of the different
distance functions which arise for different choices of the choice
functions.

\subsection{Notions of aperiodic order}
\label{ssect-charorder-AO}

After recalling some notions of aperiodic order 
we focus on the spectral triples of Section~\ref{ssec-STsubshift} and
\ref{ssec-STtrans} with the aim to derive criteria for high aperiodic
order (Section~\ref{ssect-charorder-Connesdist}).  
We will define these notions mainly for tilings of $\RM^d$.
For 
symbolic tilings (infinite words) the corresponding notions are
easily adapted.

\paragraph{Complexities and complexity exponents}
We introduced the complexity function $p(r)$ and various complexity
exponents in Section~\ref{ssect-complexity-ranks} and ~\ref{ssec-STsubshift-lapl}. 
The smaller the growth of the complexity function, the more ordered the tiling appears.
An ordered tiling is expected to have a subexponential or polynomially bounded complexity 
function.



\paragraph{Repetitiveness}
The repetitiveness function $R: \RM^+\rightarrow \RM^+$ for a tiling $T$ is defined as follows: $R(r)$ is the smallest $r'$ such that any $r'$-patch of $T$ contains an occurrence of all the $r$-patches of $T$.

We will assume that $R(r)$ is finite for all $r\ge 0$. In the context
of tilings of finite local complexity this means that
the tiling is {\em repetitive} and 
is equivalent to the minimality of the dynamical system $(\Omega, \RM^d)$.
A well ordered tiling is also supposed to have a low repetitiveness function.
The repetitiveness function is related to the complexity function in several ways. 
For instance, Lagarias and Pleasants established a bound \(R(r)\ge c p(r)^{1/d}\), for some $c>0$ and all $r$ large, which holds in general for any repetitive tiling of $\RM^d$ \cite{LP03}.
Among repetitive tilings, the {\em linearly repetitive} ones \cite{DHS99,Dur00} are commonly regarded as the most ordered \cite{LP03}.
These are tilings for which there exists a constant $c_{\text{\rm\tiny LR}}>0$ such that
\[
R(r)\le c_{\text{\rm\tiny LR}} \, r \,, \quad \forall r\ge 0\,.
\]
For linearly repetitive tilings, the reverse inequality to Lagarias--Pleasants' bound holds too \cite{Len04}, and one has: \(c \, p(r)^{1/d}\le R(r)\le c'\, p(r)^{1/d}\) for some $c,c'>0$ and $r$ large.

Finite and almost-finite ranks, defined in Section~\ref{ssect-complexity-ranks}, are related notions.
They characterize bounded or slow increasing number of repetition types.

\paragraph{Repulsiveness and bounded powers}
A return vector to a patch in a tiling is any vector joining (the punctures of) two occurrences of that patch in the tiling.
A repetitive tiling is said to be {\em repulsive} if any return to an $r$-patch grows at least like $r$, for $r$ large.
Linear repetitive tilings are repulsive \cite{Len04, BL08}, and repulsiveness is also a signature of high aperiodic order.
For example, a non-repulsive tiling has patches of arbitrarily large sizes which overlap on arbitrarily large parts.
These overlaps force local periodicity, as is easily seen in dimension $1$ and for words.

An infinite word with language $\Ll$ is repulsive if its index of repulsiveness,
\begin{equation}
\label{eq-repulsive}
\ell = \inf \Bigl\{ \frac{|W|-|w|}{|w|} \, : \, w, W \in \Ll, \, w \; \text{\rm is a proper prefix and suffix of } W
\Bigr\} \,,
\end{equation}
is positive: $\ell >0$.
This is equivalent to {\em bounded powers}: there exists an integer $p$ such that each factor occurs at most $p$ times consecutively: \(\forall u \in \Ll, u^{p+1}\notin \Ll\).

\paragraph{Equidistributed frequencies}
A uniquely ergodic tiling is said to have {\em equidistributed frequencies} if the frequency of any $r$-patch behaves like a given function of $r$, for $r$ large.
As this function is independent of the patches, it is easily related to the complexity.
Namely, a tiling has equidistributed frequencies if there are constants $c,c'>0$ such that for any $r$-patch $P$ one has:
\begin{equation}
\label{eq-equifreq}
 c\, p(r)^{-1} \le \freq(P) \le c'\, p(r)^{-1} \,.
\end{equation}
A linearly repetitive tiling has equidistributed frequencies (\cite{KS11} Theorem~1.8).

\paragraph{Uniform bound on the number of patch extensions}

In a tiling with finite local complexity, any $r$-patch $P$ has finitely many extensions: if $r=r_n$, there are finitely many $r_{n+1}$-patches containing $P$.
If this number is uniformly bounded in $n$, then the tiling is said to have a {\em uniform bound on the number of patch extensions}.
Tilings with FLC in dimension $1$, and words over a finite alphabet, obviously have such a uniform bound. This notion is closely related to finite rank, defined in Section~\ref{ssect-complexity-ranks}.

A tiling with equidistributed frequencies and for which the complexity function  satisfies $p(4r)\le c\, p(r)$, for some $c>0$ and all $r$ large, has a uniform bound on the number of patch extensions (\cite{KS11} Lemma~4.15).

\subsection{A noncommutative geometrical criterion for aperiodic order}
\label{ssect-charorder-Connesdist}

In the following we work with spectral triples for one-sided subshifts
(Def.~\ref{def-STsubshift}) and with
ordinary transversal spectral triples
for tilings or subshifts (Def.~\ref{def-STtrans}). Recall that these
spectral triples depend on various choices, namely those for the
horizontal edges and the length function, but also that for the choice
function. While we keep the choices for the
horizontal edges and the length function fixed we treat that for the
choice function $\tau$ as a parameter and thus obtain a Connes
distance $d_C = d_\tau$ which depends on this parameter. 
How are these distance functions related? In particular, 
considering the infimum and supremum of $\dtau$ over all choices
\[
\dsup = \sup_\tau \dtau\,, \qquad \dinf = \inf_\tau \dtau \,.
\]
we may ask, are $\dsup$ and $\dinf$ equivalent in the sense of
Lipschitz? By this we mean that there
exists a constant $C>0$ such that
$$ C^{-1} \dinf \leq \dsup\leq C\dinf.$$
This is the non
commutative geometrical criterion which will characterize a certain
form of aperiodic order, provided clever choices for the
horizontal edges and the length function have been made.

We first consider repetitive aperiodic tilings of finite local
complexity and their ordinary transverse spectral triples. 
We fix the {\em maximal} choice $\Hh=\Hh^{\rm max}$ for the horizontal
edges. 
Identifying the canonical 
transversal $\Xi$ of the tiling with the set of infinite paths on the tree of patches we
obtain the following expressions (\cite{KS11} Section 4.1) 
\begin{equation}
\label{eq-dConnes}
\dinf(\xi,\eta) = \delta(|\xi\wedge\eta|) \,, \qquad
\dsup(\xi,\eta) = \delta(|\xi\wedge\eta|) + \sum_{n>|\xi\wedge\eta|} \bigl( b_n(\xi) + b_n(\eta)\bigr)\; \delta(n) \,.
\end{equation}
Here $\xi\wedge\eta$ is the greatest common prefix of the paths
$\xi$ and $\eta$ and hence $|\xi\wedge\eta|$ is the radius of the 
largest $r$-patch the tilings corresponding to $\xi$ and $\eta$ have
in common (around the origin). Furthermore $b_n(\xi)=1$ if the $n$th
vertex of the path $\xi$ is a branching vertex, otherwise
$b_n(\xi)=0$. 

The first formula says that $\dinf$ corresponds to the metric defined
by $\delta$ in \eqref{ttmetric}. 
We need to assume that this function 
satisfies the following multiplicative inequalities, which are satisfied for instance for power functions: 
\begin{subequations}
\label{eq-delta}
\begin{align}
\delta(ab) \le \bar c \; \delta(a) \delta(b)
\label{eq-deltaup}\\
\delta(2a) \ge \underline{c} \; \delta(a)
\label{eq-deltalow}
\end{align}
\end{subequations}
for some constants $\bar c, \underline c >0$ and for all $a$, $b$ large.
The following result shows that Lipschitz equivalence of $\dinf$ and
$\dsup$ is a necessary criterion for equidistribution of frequencies.
\begin{theorem}[Theorem~4.16. in \cite{KS11}]
\label{theorem-charorder-tiling}
Consider a tiling which has equidistributed patch frequencies and
whose complexity function satisfies: $p(4r)\le c p(r)$ for some $c>0$
and all $r$ large. 
Suppose that 
$\delta$ satisfies the inequalities in equation~\eqref{eq-deltaup}.
Then $\dinf$ and $\dsup$ are Lipschitz-equivalent.
\end{theorem}

\medskip

We now focus on subshifts considering the spectral triples for
one-sided subshifts (Def.~\ref{def-STsubshift}) with 
{\em privileged} horizontal edges $\Hh=\Hh^{\rm pr}$. 
In other words, $u_1,u_2\in\Tt^{(0)}$ are linked with a horizontal edge 
if and only if $u_1$ and $u_2$ are privileged words, and there is a
privileged word $v\in\Tt^{(0)}$ such that $u_1$ and $u_2$ are two
distinct complete first returns to $v$ (see
Section~\ref{ssec-STsubshift}).
As  above the spectral triple requires a choice for the function $\delta$,
and will depend parametrically on the choice functions.
We consider again the infimum $\dinf$ and supremum $\dsup$ of the spectral metric
$d_C = d_{\tau}$ over all choice functions $\tau$.
These metrics, if continuous, are given on the one-sided subshift space $\Xi$ by
\begin{equation}
\label{eq-tdConnes}
\dinf(\xi,\eta) = \delta(\|\xi\wee\eta\|) \,, \qquad
\dsup(\xi,\eta) = \delta(\|\xi\wee\eta\|) + \sum_{n>\|\xi\wedge\eta\|} \bigl( \delta(\|\tilde\xi_n\|) + \delta(\|\tilde{\eta}_n\|)\bigr) \,
\end{equation}
Now $\xi\wee\eta$ denotes the greatest common {\em privileged} prefix of $\xi$
and $\eta$. Furthermore we write $\|u\|=n$ if $u$ is an $n$th order
privileged word, that is, obtained as $n$th iterated complete first
return to the empty word. Finally $\tilde{\xi}_n$ denotes the privileged
prefix of $n$th order of $\xi$. In particular,
 $\|\xi\wedge\eta\|=m$ if $\tilde\xi_m=\tilde\eta_m$ but 
$\tilde\xi_{m+1}\neq\tilde\eta_{m+1}$. 

Now Lipschitz equivalence of of $\dinf$ and $\dsup$ is a necessary and
sufficient criterion for bounded powers.
\begin{theorem}[Theorem~5.1 in \cite{KLS11}]
\label{theorem-charorder-subshift}
Consider a minimal aperiodic one-sided subshift over a finite alphabet.
Suppose that the function $\delta$ used to construct the above
spectral triples satisfies the two inequalities \eqref{eq-delta}.
Then the subshift has bounded powers if and only if $\dinf$ and $\dsup$ are
Lipschitz-equivalent. Furthermore, the corresponding two-sided
subshift has bounded powers if and only if the one-sided subshift has
bounded powers. 
\end{theorem}

To emphasise the importance of the above result we discuss
the special case of Sturmian subshifts (an example of which, the Fibonacci subshift, is given in Example~\ref{ex-Fibo}).
For these subshifts, having bounded powers is equivalent to linear repetitiveness.
Sturmian subshifts depend on a parameter $\theta$, the slope, whose
number theoretical properties are reflected in various properties of
the subshift. The following result says that our noncommutative
geometric criterion  characterizes properties of irrational numbers.
\begin{cor}
\label{cor-charorder-sturmian}
Consider a Sturmian subshift with slope $\theta$.
Suppose that the function $\delta$ used to construct the above
spectral triples satisfies the two inequalities \eqref{eq-delta}.
The following are equivalent:
\begin{enumerate}[(i)]
\item The subshift is linearly repetitive;

\item The continued fraction expansion of its slope $\theta$ has bounded coefficients;

\item The metrics $\dinf$ and $\dsup$ are Lipschitz-equivalent.
\end{enumerate}
\end{cor}
One can say even a little more: If the coefficients
in the continued fraction expansion of the slope grow very fast 
(see Theorem~4.14 in \cite{KS11} for a precise statement) then
$\dsup$ is not even continuous on the subshift space, that is, it is
not compatible with its topology. 


\section{$K$-homology}

We now discuss an application of the spectral triples for compact
ultrametric spaces to the noncommutative topology of such spaces.
According to the general theory a spectral triple over a $C^*$-algebra $A$
gives rise to a $K$-homology class of that algebra which in turn
defines a group homomorphism from the $K$-group of $A$ to the
group of integers \cite{Co94}. 
A compact ultrametric space $X$ is totally
disconnected and its $K_0$-group $K_0(C(X))$ is isomorphic to $C(X,\Z)$, 
and thus as a $\Z$-module
generated by the indicator functions on the clopen sets of $X$. 
The $K_1$-group $K_1(C(X))$ is trivial.
In this section we give an answer to the question,
which group homomorphisms $C(X,\Z)\to\Z$ can be obtained from the spectral
triples of a compact ultrametric set $X$. 

\subsection{Fredholm modules and their pairing with $K$-theory}
A spectral triple is sometimes referred to as an {\em unbounded} Fredholm
module. This suggests that it can somehow be reduced to a (bounded)
Fredholm module. 
\begin{definition}
An \emph{odd} Fredholm module $(A,F,\HS)$ 
over a $C^*$-algebra $A$ is given by a Hilbert space $\HS$, a
representation $\pi: A \rightarrow B(\HS)$, and a bounded,
self-adjoint operator $F$, such that $F^2 = 1$ and for all $a \in A$,
$[F,\pi(a)]$ is compact. 

An \emph{even} Fredholm module is an odd Fredholm module, with an
additional self-adjoint operator $\Gamma$ which satisfies $\Gamma^2 =
1$, which commutes with each operator $\pi(a)$ and anti-commutes with
$F$. 
\end{definition}
We see that the main difference to the definition of a spectral
triple is that $F$ is bounded and $F^2=1$. So it is no surprise that a
spectral triple gives rise to a Fredholm module by applying a polar
decomposition to $D$ and taking $F$ to be its unitary part\footnote{We suppose
for simplicity (and this is actually satisfied for our spectral
triples) that $D$ has no kernel.}, $D = F|D|$. Indeed, since $D$ is
self-adjoint we can obtain $F$ from $D$ by replacing the eigenvalues of
$D$ by their sign, $F = \mathrm{sign}(D)$.

As for even spectral triples, the grading operator of a Fredholm
module allows one to  decompose $\HS=\HS^+ \oplus \HS^-$, 
such that the representation acts diagonally as $\pi_+ \oplus \pi_-$, the operator $\Gamma$ acts as $1 \oplus (-1)$, and
\[
 F = \left( \begin{array}{cc}
     0 & T^* \\
     T & 0
     \end{array} \right)
\]
for some unitary operator $T:\HS^+\to \HS^-$.

Fredholm modules for $A$ are at the basis of the construction of the
$K$-homology group of $A$. We will not explain this construction here.
Our interest lies in the group homomorphisms into $\Z$ the $K$-homology classes
define on $K$-theory and these can be expressed directly on the level of
Fredholm modules. Only even Fredholm modules lead to non-trivial homomorphisms
on $K_0(A)$ and so for our purposes it is enough to consider even Fredholm modules.
\begin{definition}
 Let $F$ be a bounded operator between two Hilbert spaces. It is
 called a \emph{Fredholm operator} if the dimension of its kernel and
 of its cokernel are finite. In this case, its index $\Ind (F)$ is
 defined as 
 \[
  \Ind(F) = \dim (\ker F) - \dim (\coker F).
 \]
\end{definition}
The $K_0$-group $K(A)$ of a unital $C^*$-algebra is constructed from
homotopy classes of projections $p$ in $A$ and in $M_n(A)$, the algebra of $n\times n$ matrices with entries in $A$ (any natural $n$). 
\begin{theorem}[\cite{Co94}]
 Let $M =(A,F,\HS)$ be an even Fredholm module and $p$ be a projection in $A$.
 Then $ \pi_- (p) T  \pi_+(p)$ is a Fredholm operator between the two
 Hilbert spaces $\pi_+ (p) \HS^+$ and $\pi_-(p) \HS^-$ and 
 \begin{equation}\label{eq-pairing}
   \varphi_M(p) = \Ind \Big( \pi_- (p)\, T\,  \pi_+(p)
\Big) 
 \end{equation}
induces a group homomorphism $\varphi_M:K_0(A) \to \Z$.
\end{theorem}
We emphasize that the index has to be taken for the operator defined
between the specified source and range spaces, and not all of $\HS^+$ and all of
$\HS^-$. 

Claiming that \eqref{eq-pairing} induces a group homomorphism on the level of $K_0(A)$ means not only that it is additive, but also that $\varphi_M(p)$ is invariant under homotopy of projections, and that it extends (in a natural way) to projections in $M_n(A)$.  We remark also that there is a natural notion of direct sum for Fredholm modules and that $\phi_{M \oplus M'} = \phi_M + \phi_{M'}$. 
The above theorem may therefore be
formulated also in the following way: The map  $(M,p)
\mapsto \varphi_M(p)\in\Z$ extends to a bi-additive map
between the $K$-homology and the $K$-theory of $A$.
This bi-additive map is called the {\em Connes pairing}.

\subsection{Compact ultrametric space}
We now consider the case of a compact  ultrametric space $X$.
We have seen  in
Section~\ref{sec-ultra} that $X$ can be identified with the
space of paths on its Michon tree $\Tt = (\Tv,\Te)$ and obtain
spectral triples depending on the choice of horizontal edges $\HE$ 
and choice functions $\tau: \Tv \rightarrow \partial \Tt$. We fix 
an orientation $\HE = \HE^+ \cup \HE^-$ which induces
the grading on the Hilbert space $\HS = \HS^+ \oplus \HS^- =
\ell^2 (\Hh^+) \oplus \ell^2 (\Hh^-)$. Theorem~\ref{theorem-STbdry}
now gives a spectral triple for $C(X) \cong C(\partial\Tt)$.

Recall from Section~\ref{sec-ultra} that
the Dirac operator $D$ of the spectral triple depends on the
ultrametric. But note that $\mathrm{sign}(D)$ is independent of the that
metric and hence the Fredholm module defined by the spectral triple
depends only on $\HE$ and $\tau$. Indeed, $F = \mathrm{sign}(D)$
is given by 
\[
 F = \left( \begin{array}{cc}
     0 & T^* \\
     T & 0
     \end{array} \right),
\]
where $T : \HS^+ \rightarrow \HS^-$ is induced by changing the
orientation of the horizontal edge,
$h \mapsto h^\mathrm{op}$.
More precisely, if we denote  
$\one_h\in \HS$ the function which is $1$ on the edge $h\in\HE$ then
$T\one_h = \one_{h^\mathrm{op}}$. 
We denote the associated homomorphism of the Fredholm module by 
$\phi_{\tau,\HE} : K_0 (C(X)) \rightarrow \ZM$, that is:
\[
 \phi_{\tau,\HE} : C(X,\ZM ) \longrightarrow \ZM.
\]
By compactness of $X$, any function in $C(X;\ZM )$ is a finite sum of
the form $\sum \alpha_v \chi_v$, where $\chi_v$ is the characteristic
function of the set of all infinite paths which pass through $v$. 
In particular, the $K$-theory is spanned by projections in $C(X)$ 
rather than in a matrix algebra.
Therefore, $\phi_{\tau,\HE}$ is entirely determined by the values of
$\{ \phi_{\tau,\HE}(\chi_v) \ ; \ v \in \Vv \}$. These values can be
explicitly computed using equation~\eqref{eq-pairing}. 
Let $(\one_h)_{h \in \HE^+}$ be a basis for $\HS^+$.
Then a simple computation shows that given $h \in \HE^+$,
\(
 \pi_{-} (\chi_v) T \pi_{+} (\chi_v) \cdot \one_h = \chi_v(\tau \circ r(h)) \chi_v(\tau \circ s(h)) \cdot \one_{h^\mathrm{op}}.
\)
Therefore, the index of this linear map (between the indicated Hilbert spaces) is:
\begin{eqnarray}\label{eq-index-formula}
\varphi_{\tau,\HE}(\chi_v) &=&
      \# \{ h \in \HE^+ \ ; \ \chi_v(\tau \circ s (h)) \neq 0 \text{
        and } \chi_v(\tau \circ r (h)) = 0 \} \\ 
    &&- \# \{ h \in \HE^+ \ ; \ \chi_v(\tau \circ r (h)) \neq 0 \text{ and } \chi_v(\tau \circ s (h)) = 0 \}.\nonumber
\end{eqnarray}
The same arguments as in the proof for Lemma~\ref{lem-dense} show
that whenever $h \in \HE^+$ does not satisfy that
either $r(h)$ or $s(h)$ lies on the unique path from the root to $v$,
then $h$ does not belong to any of the sets occuring in
equation~\eqref{eq-index-formula}. 
This equation can therefore be rewritten:
\begin{equation}\label{eq-pairing-explicit}
\varphi_{\tau,\HE}(\chi_v) = \sum_{h \in \HE^+}\Big( \chi_v (\tau \circ s(h)) - \chi_v(\tau \circ r(h)) \Big), 
\end{equation}
the general term of the sum being non-zero only for finitely many $h$.

In particular, it results from these formulas that for any $\tau$ and
any choice of minimal edges $\HE=\HE^\mathrm{min}$ we must have
$\phi_{\tau,\HE} (1) = 0$, and that there are two vertices $v_1,
v_2\in\Tt^{(0)}_1$, such that $\phi_\tau (\chi_{v_1})
= 1$ and $\phi_\tau (\chi_{v_2}) = -1$. This provides the following
proposition.  

\begin{prop}
 For any choice function $\tau$ and any choice of minimal edges $\HE^{\rm min}$,
 the homomorphism
$\varphi_{\tau,\HE}$ is non-trivial.
 In particular, the $K$-homology class of a Pearson-Bellissard
 spectral triple is never trivial. 
\end{prop}

Now, the question is: which homomorphisms on $K_0(C(X))$ can we obtain
from Fredholm modules? The following proposition says that,
if we only slightly relax our earlier assumptions by assuming that $\HE$ can
consist of several edges between the same vertices, then the 
condition found above, namely $\phi(1) = 1$, is the only obstruction for an element
of $\Hom (K_0(C(X));\ZM)$ to come from a Fredholm module.

\begin{prop}\label{prop-Fred}
For any $\phi \in \Hom (C(X;\ZM);\ZM)$ such that $\phi (1) = 0$, there is a set of horizontal edges $\HE$(possibly bigger than what was defined as a maximal set), and a choice function $\tau$ such that  
$\phi_{\tau,\HE} = \phi$.
\end{prop}

\begin{proof}
We begin by a simple remark. Suppose that we have two modules defined by two distinct choices $\tau,\HE$ and $\tau',\HE'$. If the horizontal edges coincide up to level $n$ and the choice functions coincide up to level $n-1$ then $\varphi_{  \tau,\HE} (\chi_v) = \varphi_{  \tau',\HE_2} (\chi')$ for all vertices $v$ of level smaller or equal to $n$. We can therefore construct inductively 
$\tau,\HE$ from the values of $\varphi(\chi_v)$ to obtain $\varphi_{  \tau,\HE} = \varphi$ as follows.
%

Let $v_0\in\Tt^{(0)}_0$ be the root. Choose $\tau (v_0)$ arbitrarily.
Let $v_1, \ldots, v_k\in\Tt^{(0)}_1$.
Suppose that $v_1$ is the vertex through which $\tau(v_0)$ passes. 
Then we must have $\tau (v_1) = \tau(v_0)$. For $i \neq 0,1$ we may choose  $\tau(v_i)$ arbitrarily.
We now choose the horizontal edges $\HE_1$ of level $1$, c.f.\  Figure~\ref{Khom:trees}. 
Let $n_1, \ldots, n_k$ be the numbers $n_i := \phi (\chi_{v_i})$.
Since $\sum_i \chi_{v_i} = 1$ and $\phi (1) = 0$, the $n_i$ sum to $0$.
By convention, we say that ``$n$ horizontal edges from $v_i$ to $v_j$'' consists of $|n|$ edges with source $v_i$ and range $v_j$ if $n$ is positive, and $|n|$ edges with range $v_i$ and source $v_j$ if it is negative.
Then $\HE_1^+$ shall consist of $n_1$ edges from $v_1$ to $v_2$, $n_1+n_2$ edges from $v_2$ to $v_3$, \ldots, and $\sum_{i=1}^{k-1} n_i$ edges from $v_{k-1}$ to $v_k$.
A simple computation shows that however we extend $\tau$ and whatever the choice of horizontal edges of higher level, $\phi_{\tau,\HE} (\chi_v)$ assumes the correct values.

To proceed with the inductive construction, assume that the choice function and horizontal edges are determined up to level~$n-1$ and $n$, respectively. Let $w_0 \in \Tv_n$ and 
$w_1, \ldots, w_k\in\Tt^{(0)}(w_0)$. One may assume that $\tau(w_0)$ passes through $w_1$. Then we must have $\tau(w_1)=\tau(w_0)$ and we choose $\tau(w_i)$ arbitrarily for the other vertices.
Let $m_i:=\phi(\chi_{w_i})$ for $i=0, \ldots, k$.
Then,  however we extend $\tau$ further and whatever the choice of horizontal edges of level $n+1$
we have  $\phi_{\tau,\HE} (w_0) = \phi (w_0) = m_0$.
Now $\HE_{n+1}^+$ should consist of 
$-m_0+\sum_{i=1}^{k-1}m_i$ edges from $w_{k-1}$ to $w_k$.
Since $\sum_{i=1}^k m_i = m_0$ we have  $\phi_{\tau,\HE} (\chi_{w_i}) = \phi(\chi_{w_i})$ for the other $i$ as well. 

The construction of the last paragraph is iterated for all vertices at level $n$, and then for all levels.
\end{proof}

 \begin{figure}[ht]
  \begin{center}
 \includegraphics[scale=1]{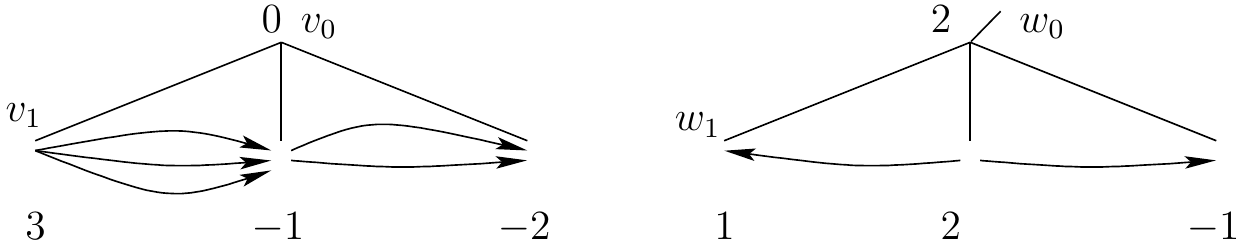}
   \caption{{\small The left picture shows the beginning of the tree with its root vertex $v_0$. 
The right figure shows a part of the tree corresponding to a branching at vertex $w_0$ of some higher level.
In both cases we assume that the choice function applied to the top vertex corresponds to a path via the left vertical edge.
The numbers assigned to vertices stand for the values that $\phi$ assigns to the indicator functions corresponding to the vertex. 
The horizontal edges are chosen according to the rule explained in the proof of 
Proposition~\ref{prop-Fred}.}}
   \label{Khom:trees}
  \end{center}
 \end{figure}

\begin{rem}
 It could happen that the construction above provides a Fredholm operator for which some of the $\HE_n$ are empty. It still defines a Fredholm module, and adding appropriate weights defines a spectral triple.
A Fredholm module can also be modified in order to augment the sets $\HE$ without changing the pairing map: given two vertices $v$ and $v'$, simply add an edge from $v$ to $v'$ and an edge from $v'$ to $v$.
\end{rem}


\end{document}